\documentclass[11pt,a4paper]{article} 
\usepackage[utf8]{inputenc}
\usepackage[english]{babel}
\usepackage[T1]{fontenc}
\usepackage{amsmath}
\usepackage{amsthm}
\usepackage{comment}
\usepackage{amsfonts}
\usepackage{amssymb}
\usepackage{multicol}

\usepackage{url}
\usepackage{pdfpages}

\usepackage{graphicx}
\usepackage{tikz-qtree}
\usepackage[totoc]{idxlayout}

\usepackage[intoc,refpage]{nomencl}

\usepackage{graphicx}
\graphicspath{ {images/} }
\usepackage[all,cmtip]{xy}

\usepackage[toc,page]{appendix}
\usepackage[left=2cm,right=2cm,top=3cm,bottom=3cm]{geometry}
\usepackage{pgf,tikz}
\usepackage[font=small,skip=0pt]{caption}

\usetikzlibrary{arrows}

\makeatletter
\newcommand\ackname{Acknowledgements}
\if@titlepage
   \newenvironment{acknowledgements}{%
       \titlepage
       \null\vfil
       \@beginparpenalty\@lowpenalty
       \begin{center}%
         \bfseries \ackname
         \@endparpenalty\@M
       \end{center}}%
      {\par\vfil\null\endtitlepage}
\else
   \newenvironment{acknowledgements}{%
       \if@twocolumn
         \section*{\abstractname}%
       \else
         \small
         \begin{center}%
           {\bfseries \ackname\vspace{-.5em}\vspace{\z@}}%
         \end{center}%
         \quotation
       \fi}
       {\if@twocolumn\else\endquotation\fi}
\fi
\makeatother

\def\restriction#1#2{\mathchoice
              {\setbox1\hbox{${\displaystyle #1}_{\scriptstyle #2}$}
              \restrictionaux{#1}{#2}}
              {\setbox1\hbox{${\textstyle #1}_{\scriptstyle #2}$}
              \restrictionaux{#1}{#2}}
              {\setbox1\hbox{${\scriptstyle #1}_{\scriptscriptstyle #2}$}
              \restrictionaux{#1}{#2}}
              {\setbox1\hbox{${\scriptscriptstyle #1}_{\scriptscriptstyle #2}$}
              \restrictionaux{#1}{#2}}}
\def\restrictionaux#1#2{{#1\,\smash{\vrule height .8\ht1 depth .85\dp1}}_{\,#2}}

\newcommand*{\quantifies}{\raisebox{\depth}{\scalebox{-1}[-1]{$\mathsf{Q}$}}}

\DeclareMathOperator{\tp}{tp}

\DeclareMathOperator{\val}{val}
\DeclareMathOperator{\res}{res}
\DeclareMathOperator{\Res}{Res}
\DeclareMathOperator{\ac}{ac}

\DeclareMathOperator{\rv}{rv}
\DeclareMathOperator{\RV}{RV}

\DeclareMathOperator{\Th}{Th}

\author{Pierre Touchard\footnote{
The author was funded by the Deutsche Forschungsgemeinschaft (DFG, German Research Foundation)  through SFB 878 and under Germany ’s Excellence Strategy – EXC 2044 – 390685587, Mathematics Münster: Dynamics – Geometry - Structure.  This research has been also supported by the DAAD through the `Kurzstipendien für Doktoranden 2020/21' programme.}
}
\title{Stably Embedded Submodels of Henselian Valued Fields }

\theoremstyle{plain}
\newtheorem{theorem}{Theorem}[section]
\newtheorem{lemma}[theorem]{Lemma}
\newtheorem{proposition}[theorem]{Proposition}
\newtheorem{claim}{Claim}
\newtheorem{remark}[theorem]{Remark}

\newtheorem{observation}[theorem]{Observation}
\newtheorem{fact}[theorem]{Fact}

\theoremstyle{definition}
\newtheorem{definition}[theorem]{Definition}
\newtheorem{corollary}[theorem]{Corollary}

\theoremstyle{remark}
\newtheorem*{remark*}{Remark}

\newtheorem*{notation}{Notation}
\newtheorem*{question}{Question}
\definecolor{black}{rgb}{0,0,0}
\newtheorem*{example}{Example}
\newtheorem*{examples}{Examples}

\begin{document}

\maketitle

\begin{abstract}
    We show a transfer principle for the property that all types realised in a given elementary extension are definable. It can be written as follows: a Henselian valued fields is stably embedded in an elementary extension if and only if its value group is stably embedded in its corresponding extension, its residue field is stably embedded in its corresponding extension, and the extension of valued fields satisfies a certain algebraic condition. We show for instance that all types over the Hahn field $\mathbb{R}((\mathbb{Z}))$ are definable. Similarly, all types over the quotient field of the Witt ring $W(\mathbb{F}_p^{\text{alg}})$ are definable. This extends a work of Cubides and Delon and of Cubides and Ye.
\end{abstract}
\tableofcontents

\newpage

 \section*{Introduction}
    In model theory of unstable structures, the notion of stable embeddedness plays an important role. A subset $A$ of an $\mathrm{L}$-structure $\mathcal{M}$ is said to be stably embedded if any intersection of $A$ and of a definable set is the intersection of $A$ and of an $A$-definable set. By a duality between parameters/variables, it is equivalent to say that all types over $A$ are definable over $A$. Understanding stably embedded definable subsets can be a crucial step in order to understand the whole structure. For instance, in the study of Henselian valued fields of equicharacteristic $0$, the value group and residue field are stably embedded substructures that play a primordial role. Another aspect concerns stably embedded submodels of the structure. Given a theory $T$, one can ask: when is a model $\mathcal{M}$ stably embedded in a given elementary extension $\mathcal{N}$? in all elementary extensions?
    By an observation of van den Dries (\cite{VdD86}), being a stably embedded submodel of a real closed field is a first order property in the language of pairs. Motivated by a similar question, Cubides and Delon have shown in \cite{CD16} that an algebraically closed valued field $\mathcal{K}$ is stably embedded in an elementary extension  $\mathcal{L}$ if and only if the valued field extension $\mathcal{K} \preceq \mathcal{L}$ is \textit{separated} (see Definition \ref{DefinitionSeparated}) and the small value group $\Gamma_K$ is stably embedded in the larger value group $\Gamma_L$. Recently, Cubides and Ye proved in \cite{CY19} a similar statement for p-adically closed valued fields and real closed valued fields. In this paper, we give a generalisation to the following (incomplete) theories of Henselian valued fields, which we call here \textit{benign} theories:
    \begin{enumerate}
        \item Henselian valued fields of characteristic $(0,0)$,
        \item algebraically closed valued fields, 
        \item algebraically maximal Kaplansky Henselian valued fields,
    \end{enumerate}
    In fact, we give two statements:

{
    \renewcommand{\thetheorem}{\ref{TheoremReductionGammakElementaryPairs}}
    \begin{theorem}
    Assume that $T$ is a benign theory of Henselian valued fields. 
    Let $\mathcal{K} \preceq \mathcal{L}$ be an elementary pair of models of $T$. The following are equivalent:
        \begin{enumerate}
            \item $\mathcal{K}$ is stably embedded in $\mathcal{L}$,
            \item The extension $\mathcal{L}/\mathcal{K}$ is separated, the residue field of $\mathcal{K}$ is stably embedded in the residue field of $\mathcal{L}$ and the value group of $\mathcal{K}$ is stably embedded in the value group of $\mathcal{L}$.
        \end{enumerate}
    \end{theorem}
    \addtocounter{theorem}{-1}
    }   
     Even if in practice, we will mainly consider elementary extensions of benign valued fields as the situation is simpler, the second statement concerns non elementary extensions of valued fields: 
    
    {
\renewcommand{\thetheorem}{\ref{TheoremReductionGammakNonElementaryPairs}}
    \begin{theorem}
    Assume that $T$ is a benign theory of Henselian valued fields. 
    Let $\mathcal{L}/ \mathcal{K}$ be a separated extension of valued fields with $\mathcal{L}\models T$. 
    Assume either
    \begin{itemize}
            \item that the value group of $\mathcal{K}$ is a pure subgroup of the value group of $\mathcal{L}$,
            \item or that the multiplicative group of the residue field of $\mathcal{L}$ is divisible. 
    \end{itemize}
    
    The following are equivalent:
    
        \begin{enumerate}
            \item $K$ is stably embedded in $\mathcal{L}$,
            \item the residue field of $\mathcal{K}$ is stably embedded in the residue field of $\mathcal{L}$ and the value group of $\mathcal{K}$ is stably embedded in the value group of $\mathcal{L}$.
        \end{enumerate}
    \end{theorem}
\addtocounter{theorem}{-1}
}

    In particular, the field of $p$-adics $\mathbb{Q}_p$ is stably embedded in $\mathbb{C}_p$, the completion of its algebraic closure. With some additional observation, one can adapt these arguments and treat the case of unramified mixed characteristic Henselian valued fields with perfect residue field. We show indeed a similar statement for elementary pairs:

{
\renewcommand{\thetheorem}{\ref{TheoremReductionGammakElementaryPairsMixeCharacteristic}}
    \begin{theorem} 
        Let $\mathcal{K}$ be a unramified mixed characteristic Henselian valued field with perfect residue field and $\mathcal{L}$ be an elementary extension. The following are equivalent:
        \begin{enumerate}
            \item $\mathcal{K}$ is stably embedded (resp.\ uniformly stably embedded) in $\mathcal{L}$;
            \item The extension $\mathcal{L}/\mathcal{K}$ is separated, the residue field of $\mathcal{K}$ is stably embedded (resp.\ uniformly stably embedded) in the residue field of $\mathcal{L}$ and the value group of $\mathcal{K}$ is stably embedded (resp.\ uniformly stably embedded) in the value group of $\mathcal{L}$.
        \end{enumerate}
    \end{theorem}
    \addtocounter{theorem}{-1}
}

    Here is an overview of the present paper. Section \ref{SectionPreliminaries} is dedicated to preliminaries: we first define two notions of stable embeddedness (uniform and non-uniform) before discussing some easy cases of stably embedded elementary pairs. We will also illustrate these notions with a digression on the theory of the random graph in Appendix \ref{OnPairsOfRandomGraphs}. In Subsection \ref{SectionModelTheoryAlgebraicStructures}, we recall some relative quantifier elimination results and other relevant properties that benign Henselian valued fields share. Then, we prove the theorems cited above in Section \ref{SectionStablyEmbeddedPairs}. The demonstration for benign valued fields is in Subsection \ref{SectionReductionStableEmbeddednessBenignValuedFields}. It includes a similar reduction for submodels of short exact sequences of abelian groups. In Subsection \ref{SectionReductionStableEmbeddednessUnramifiedMixedCharacteristic}, we treat the case of elementary extension of unramified mixed characteristic Henselian valued fields with perfect residue field, using again the result on short exact sequences of abelian groups. Finally, in Subsection \ref{SectionPairs}, we discuss elementary pairs, in particular elementary pairs of valued fields (with a predicate for the smaller model), and partially answer to the following question: when is the class of stably embedded elementary pairs of Henselian valued fields axiomatisable?

\section{Preliminaries} \label{SectionPreliminaries}

 We will assume the reader to be familiar with basic model theory concepts, and in particular with standard notations. One can refer to \cite{TZ12}. Symbols $x,y,z,\dots$ will usually refer to tuples of variables, $a,b,c,\dots$ to parameters. Capital letters $K,L,M,N,\dots$ will refer to sets, and calligraphic letters $\mathcal{K},\mathcal{L},\mathcal{M},\mathcal{N},\dots$ will refer to structures with respective base sets $K,L,M,N,\dots$. If there is no ambiguity, we may respectively name a very saturated elementary extension with blackboard bold letters $\mathbb{K},\mathbb{L},\mathbb{M},\dots$ Languages will be denoted will a roman character $\mathrm{L},\mathrm{L}',\mathrm{L}_{Rings}, \mathrm{L}_{\Gamma,k}$ etc. 
 We will also use Rideau-Kikuchi's terminology about enrichment and resplendence. The reader can refer to \cite[Annexe A]{Rid17} for a detailed exposition. We will also use notation and some definitions used in \cite{Tou20a}.

 

\subsection{Stable embeddedness}\label{SubsecStableEmbeddedness}
    As Cubides, Delon and Ye, we will consider the question of stably embedded elementary pairs of certain valued fields and establish a transfer principle. We will often use notations from \cite{CD16}, that we will recall. As we do not only consider definable stably embedded subsets, it appears natural to distinguish two notions of stable embedddness. Fortunately, this distinction will not bring any difficulties: the same proofs for our transfer principle holds \textit{mutatis mutandis} for the uniform and non-uniform case.

    Let $\mathrm{L}$ be any first order language, and $T$ be an $\mathrm{L}$-theory. 
    
    \begin{definition} [stable embeddedness for arbitrary sets]\label{DefinitionStableEmbeddness}
        Let $\mathcal{M}$ be an $\mathrm{L}$-structure. A subset $S\subseteq M$ is said to be \textit{stably embedded in} $\mathcal{M}$\index{Stable embeddedness! non-uniform} if for every formula $\phi(x,y)$ and for every tuple of parameters $a\in M^{\vert y \vert}$, there is another formula $\psi(x,z)$ and a tuple of parameters $b\in S^{\vert z \vert}$ such that $\phi(S^{\vert x\vert},a)=\psi(S^{\vert x\vert},b)$. In this case, we write $S \subseteq^{st} \mathcal{M}$ \nomenclature[SE]{$S \subseteq^{st} \mathcal{M}$}{}.
    \end{definition}

    Notice that $\psi(x,z)$ may depend on the parameter $a\in M$. For this reason, this definition is not the usual one. It results from this `non-uniformity' the following: 
    
    \begin{remark}\label{Rmkstnonuniform}
        Let $\mathcal{M}\preceq \mathcal{M}'$ be two $\mathrm{L}$-structures, and $D(x)$ be a $\mathrm{L}$-formula with $\vert x \vert=1$. Assume that $D(M')$ is stably embedded in $\mathcal{M}'$. Then $D(M)$ is stably embedded in $\mathcal{M}$. However, the converse does not always hold i.e., it is possible to have $D(M)$ stably embedded in $\mathcal{M}$, but $D(M')$ to fail to be stably embedded in $\mathcal{M}'$ (see example below). In other words, the property of \emph{stable embeddedness} of a definable set is not in general preserved by the elementary extension relation.
    \end{remark}
    \begin{example}
        We consider the atomic Boolean algebra $\mathcal{M}=\{\mathcal{P}_f(\omega) \cup \mathcal{P}_c(\omega), \cup, \cap, \cdot^c, 0, 1\}$, where $\mathcal{P}_f(\omega)$ is the set of finite subsets of $\omega$ and $\mathcal{P}_c(\omega)$ is the set of cofinite subsets of $\omega$. We refer to the partial order of inclusion as \emph{majoring}. By \cite{Poi00}, one has quantifier elimination if we add unary predicates $A_n$, $n\in \mathbb{N}$, for the elements majoring exactly $n$ atoms (here, a predicate for the subsets of $\omega$ of $n$ elements). Then, one sees that the set of atoms $A_1^{\mathcal{M}}$, is stably embedded in $\mathcal{M}$. Indeed, all definable subsets of $A_1^{\mathcal{M}}$ are finite or cofinite and more generally, definable subsets of $(A_1^{\mathcal{M}})^n$ are finite Boolean combination of products of such sets and diagonals.  However in an $\aleph_0$-saturated elementary extension $\mathcal{M}'$, we have an element $f$ which is majoring infinitely many atoms and not majoring infinitely many other atoms. This allows us to define an infinite and co-infinite subset of $A_1^{\mathcal{M}'}$, which naturally cannot be defined using only parameters in $A_1^{\mathcal{M}'}$.
    \end{example}
    
    \begin{definition}
        Let $S$ be a subset of an $\mathrm{L}$-structure $\mathcal{M}$, and $a\in M^{\vert x \vert}$ a tuple of elements. We say that the type $p(x)=\tp(a/S)$ is \emph{definable} if for every formula $\phi(x,y)$, there is an $\mathrm{L}_{S}$-formula $\psi(y,b)$ such that for all $s\in S$ \[p(x) \vdash \phi(x,s) \text{ if and only if } \mathcal{M} \models \psi(s,b).\] 
    \end{definition}
    
   Let us recall some notations from \cite{CD16}. Notice that it has a slightly more general meaning, as we also consider non-elementary extensions $\mathcal{M}\subseteq \mathcal{N}$.
   
    \begin{notation}\label{NotationDK}
        Let $\mathcal{N}$ be an $\mathrm{L}$-structure and $S$ be a subset. We write 
    $T_n(S,\mathcal{N})$\nomenclature[SE]{$T_n(S,\mathcal{M})$}{} if all $n$-types over $S$ (for the theory of $\mathcal{N}$) realised in $\mathcal{N}$ are definable. If $\mathcal{M}$ is an elementary substructure of $\mathcal{N}$, we might write $T_n(\mathcal{M},\mathcal{N})$ (both in curvy letters) instead of $T_n(M,\mathcal{N})$ to emphasise it. 
    We write $T_n(\mathcal{M})$ if all $n$-types over $\mathcal{M}$ (in any elementary extension) are definable.
    \end{notation}

    One sees immediately the following fact:
    
    \begin{fact}\label{FactDualitySEDefinabilityTypes}
    A set $S$ is stably embedded in a structure $\mathcal{M}$ if and only if $T_n(S,\mathcal{M})$ holds for all $n \in \mathbb{N}$. 
    \end{fact}

    We give now a natural `uniform' version of the notion of stable embeddedness and that of definable types (Definitions  \ref{DefUniStaEmb} and \ref{DefUniDefTyp}). 
    
    \begin{definition}\label{DefUniStaEmb}
        A subset $S$ of a structure $\mathcal{M}$ is said to be \textit{uniformly stably embedded in} $\mathcal{M}$\index{Stable embeddedness! uniform} if for every formula $\phi(x,y)$, there is a formula $\psi(x,z)$ such that: 
        \begin{align} \tag*{$(\star)$} 
         \text{for all $a\in M^{\vert y \vert}$, there is a $b\in S^{\vert z \vert}$ such that $\phi(S^{\vert x\vert},a)=\psi(S^{\vert x\vert},b)$. }
		\end{align}         
         
          In that case, we write $S \subseteq^{ust} \mathcal{M}$. \nomenclature[SE]{$S \subseteq^{ust} \mathcal{M}$}{}\\

    \end{definition}

    When $S$ is a definable subset of $\mathcal{M}$, $(\star)$ is a first order property of the formula $\phi(x,y)$ and $\psi(x,z)$. So it is in particular preserved by elementary extension. It then makes sense to say that a formula is \textit{uniformly stably embedded} in a given complete theory $T$. Following the usual convention, we will omit to specify `uniform' in such a context. The reason is:\\
    
      \begin{remark}\label{RmqUstMonsterModel}
        Assume $T$ is complete. Let $D(x)$ be an $\mathrm{L}$-formula. The following statements are equivalent:
        \begin{enumerate}
            \item $D(M)$ is uniformly stably embedded in every model $\mathcal{M}$ of $T$,
            \item $D(M)$ is uniformly stably embedded in some model $\mathcal{M}$ of $T$,
            \item $D(M)$ is stably embedded in every model $\mathcal{M}$ of $T$,
            \item $D(\mathbb{M})$ is stably embedded in an $\vert \mathrm{L}\vert$-saturated model $\mathbb{M}$ of $T$.

        \end{enumerate}
    \end{remark}
    
    \begin{proof}
        $(1)\Rightarrow (2)$ and  $(1) \Rightarrow (3) \Rightarrow (4)$ are obvious. $(2) \Rightarrow (1)$ follows from the fact that $(\star)$ is first order. It remains to prove $(4)\Rightarrow (2)$. It immediately follows from compactness, but we give few details here. We have to show that         \[D(\mathbb{M})\subseteq^{st} \mathbb{M} \Rightarrow D(\mathbb{M})\subseteq^{ust} \mathbb{M},\]
        for an $\vert \mathrm{L}\vert$-saturated model $\mathbb{M}$ of $T$. Take such a model $\mathbb{M}$. If $\vert D(\mathbb{M})\vert < 2$, there is nothing to do. Assume that there is an $\mathrm{L}$-formula $\phi(x,y)$ such that for all finite sets $\Delta$ of $\mathrm{L}$-formulas $\psi(x,z)$, one has
        \[\mathbb{M} \models \exists b_\Delta \ \bigwedge_{\psi(x,z)\in \Delta} \forall c\in D \ \exists a\in D  \ \phi(a,b_\Delta) \nLeftrightarrow \psi(a,c). \]
        By compactness and $\vert \mathrm{L}\vert$-saturation, there is an element $b\in \mathbb{M}$ such that $\phi(D(\mathbb{M}),b)$ is not $\mathrm{L}(D(\mathbb{M}))$-definable. This a contradiction. 
        We have shown that there is a finite set $\Delta$ of $\mathrm{L}$-formulas such that
        \[ \mathbb{M} \models \forall b \ \bigvee_{\psi(x,z)\in \Delta} \exists c\in D\ \forall a\in D \ \phi(a,b) \Leftrightarrow \psi(a,c).\]
        Since $\vert D(\mathbb{M})\vert \geq 2$, one can use new parameters to encode all formulas in $\Delta$ in a single one. In other words, we got an $\mathrm{L}$-formula $\Psi(x,z')$ as wanted:
        \[ \mathbb{M} \models \forall b \  \exists c'\in D\ \forall a\in D \ \phi(a,b) \Leftrightarrow \Psi(a,c').\qedhere\]
    \end{proof}
    
    \begin{definition}
        We say that an $\mathrm{L}$-formula $D(x)$ is \textit{(uniformly) stably embedded} for $T$ if $D(x)$ satisfies one (equivalently any) of the conditions in Remark \ref{RmqUstMonsterModel} above.
    \end{definition}
    
     \begin{definition}\label{DefUniDefTyp}
        Let $S$ be a subset of an $\mathrm{L}$-structure $\mathcal{M}$. We say that the family $(\tp(a/S) )_{ a\in M^{\vert x \vert }}$ of all types over $S$ realised in $\mathcal{M}$  is \emph{uniformly definable} if for every formula $\phi(x,y)$, there is an $\mathrm{L}$-formula $\psi(y,z)$ such that for every tuple  $a\in M^{\vert x\vert}$, there is a tuple $b\in S^{\vert z\vert}$ such that for all $s\in S$: 
        \[\tp(a/S) \vdash \phi(x,s) \text{ if and only if } \mathcal{M} \models \psi(s,b).\] 
    \end{definition}
    We use again the notations of \cite{CD16}, but adapted to this uniform definition.
    
    \begin{notation}\label{NotationUniformeDK}
        Let $\mathcal{N}$ be an $\mathrm{L}$-structure and let $S$ be a subset. For $n\in \mathbb{N}$, we write $T_n^u(S,\mathcal{N})$\nomenclature[SE]{$T_n^u(S,\mathcal{M})$}{} if the family of types $(\tp^\mathcal{N}(a/S))_{a \in N^n}$ realised in $\mathcal{N}$ is uniformly definable. If $\mathcal{M}$ is an elementary substructure of $\mathcal{N}$, we might write $T_n^u(\mathcal{M},\mathcal{N})$ (both in curvy letters) instead of $T_n^u(M,\mathcal{N})$ to emphasise it. 
        We write $T_n^u(\mathcal{M})$ if all $n$-types over $\mathcal{M}$ (in any elementary extension) are uniformly
            definable.
    \end{notation}
    
    Similarly to the non-uniform case, one can also give a characterisation of uniform stable embeddedness in term of uniform definability of types: 
    \begin{fact}
    A set $S$ is uniformly stably embedded in a structure $\mathcal{M}$ if and only if $T_n^u(S,\mathcal{M})$ holds for all $n \in \mathbb{N}$. 
    \end{fact}
    \begin{example}
        In every benign theory of Henselian valued fields (see the list in the introduction), the residue field $k$ and the value group $\Gamma$ are stably embedded. This is a well known corollary of relative quantifier elimination in the language augmented by angular components (see Paragraph \ref{Preliminaries Benign theory of Henselian valued fields}). In fact, they are pure with control of parameters in the sense of \cite[Definition 1.8]{Tou20a}. 
    \end{example}
    
    We are interested in the following question: given a substructure $\mathcal{M}$  of an $\mathrm{L}$-structure $\mathcal{N}$, when is $M$ stably embedded (resp.\ uniformly stably embedded) in $\mathcal{N}$?
    The following (important) remark is immediately deduced from the stable embeddedness/ definability of types duality. It will be implicitly used in the remaining of this text. 
    
    \begin{remark}\label{RmkTeq} 
        Let $\mathcal{N}$ be an $\mathrm{L}$-structure, and $M$ a subset of $N$ and $S$ an interpretable sort of $\mathcal{N}$. We denote by $S(M)$ the image of $M$ under the projection of $N$ to $S$. 
        \begin{itemize}
            \item If $M \subseteq^{st} \mathcal{N}$ (resp.\ $M \subseteq^{ust} \mathcal{N}$) holds, then $S(M) \subseteq^{st} \mathcal{N}^{eq}$ (resp.\ $S(M) \subseteq^{ust} \mathcal{N}^{eq}$) holds.
            \item If $\mathcal{M}$ is an elementary submodel of $\mathcal{N}$, then 
            $\mathcal{M} \subseteq^{st} \mathcal{N}$ (resp.\ $\mathcal{M} \subseteq^{ust} \mathcal{N}$) holds if and only if $\mathcal{M}^{eq} \subseteq^{st} \mathcal{N}^{eq}$ (resp.\ $\mathcal{M}^{eq} \subseteq^{ust} \mathcal{N}^{eq}$) holds.
        \end{itemize}
    \end{remark}
    
    In other words, one can freely add some imaginary sorts to the language, or conversely remove them from the language. For instance, a theory of Henselian valued fields can either be described in the language $\mathrm{L}_{\RV}$ or the language $\mathrm{L}_{\Gamma,k}$ (see Paragraph \ref{Preliminaries Benign theory of Henselian valued fields}), and the question of stably embedded subsets will not be affected by this choice. For this reason, we will use different languages across the sections (notably in Paragraphs \ref{SectionRV} and \ref{SectionGammaK}).
    However, adding new structure might change the notion of stably embedded substructures. The following is clear:
    
    \begin{example}
        The set of rationals $\mathbb{Q}$ is uniformly stably embedded in $(\mathbb{Q}^{rc} ,0,1,+,\cdot,<)$, its real closure. However, it is not stably embedded in $(\mathbb{Q}^{rc},0,+,<)$ as an ordered abelian group (consider the cut in  $\sqrt{2}$). Of course it is uniformly stably embedded again in $\mathbb{Q}^{rc}$, as a pure set. 
    \end{example}
    
    In our study of Henselian valued fields, the question of stably embedded sub-valued fields can be asked in the language of valued fields enriched with angular components. We will be able to treat this question as well in Paragraph \ref{SubSecApplAC}. One has to notice also that the set of stably embedded subsets is in general not closed under definable  closure, as shown in the following example:
    \begin{example}
        The set of integers $\mathbb{Z}$ is uniformly stably embedded in the ordered abelian group $(\mathbb{R},0,+,<)$, but its definable closure $\mathbb{Q}$ is not.  
    \end{example}
       \paragraph{\textbf{More on definability of types}}\label{PreliminariesOnDefinableTypes}
       In order to apply Theorems \ref{TheoremReductionGammakNonElementaryPairs} and \ref{TheoremReductionGammakElementaryPairs}, one has to understand stably embedded substructures in ordered abelian groups and fields. We will focus in this text on basic examples, namely on o-minimal theories and on Presburger arithmetic. The reader will also find in Appendix \ref{OnPairsOfRandomGraphs} a similar discussion on the random graph where we also construct a uniform stably embedded pair and a non-uniform stably embedded pair of random graphs.

       Recall that $\mathrm{L}$ is any first order language. We saw that a substructure is stably embedded if and only if all realised types over this structure are definable. As we will see, in certain cases it is actually enough to show that $1$-types are definable. We will call this property the `\emph{Marker-Steinhorn criterion}', since it was first proved for elementary pairs of o-minimal structures by Marker and Steinhorn. This has been studied in the last two decades, and counterexamples to natural generalisations of the Marker-Steinhorn criterion have been found. We treat here only basic examples, but the reader will find a better overview in \cite{CD16} for instance. \\

       \textbf{O-minimal theories}
       
    \begin{fact}[Marker-Steinhorn (\cite{MS94})]\label{FactMS}
        Let $T$ be an $o$-minimal theory, and let $\mathcal{M}\preceq \mathcal{N}$ be two models. Then for all $n\in \mathbb{N}$, $T_1(\mathcal{M},\mathcal{N}) \Rightarrow T_n(\mathcal{M},\mathcal{N})$.
    \end{fact}

    From the proof of this fact follows a uniform version of it:  
    if $T$ is o-minimal, then 
    \[T_1^u(\mathcal{M},\mathcal{N}) \Rightarrow T_n^u(\mathcal{M},\mathcal{N}).\]
    One can deduce this from the non-uniform theorem by a general argument: let us add a predicate $P$ for the small model $\mathcal{M}$ to the language. Let $(\mathbb{N},\mathbb{M})$ be an $\vert \mathrm{L} \vert$-saturated elementary extension of $(\mathcal{N},\mathcal{M})$. As it is first order, we also have $T_1^u(\mathbb{M,\mathbb{N}})$ in the language $\mathrm{L}$ (however, it doesn't have to be true in the language $\mathrm{L}_P= \mathrm{L} \cup \{P\}$). In particular we have $T_1(\mathbb{M,\mathbb{N}})$ and then by Marker-Steinhorn, $T_n(\mathbb{M,\mathbb{N}})$ for all $n$. Using that $(\mathbb{N},\mathbb{M})$ is $\vert \mathrm{L}\vert $-saturated and following the proof of Remark \ref{RmqUstMonsterModel}, we get that $T_n^u(\mathbb{M,\mathbb{N}})$ holds for all $n$.  By elementarity, we have $T_n^u(\mathcal{M,\mathcal{N}})$ as wanted. As an immediate consequence, one deduces a previous result of Van den Dries: all types over an o-minimal expansion of $\mathbb{R}$ are definable, as the only possible cuts in $\mathbb{R}$ are of the form $a_-,a_+,+\infty$  or $-\infty$.

    \textbf{Presburger arithmetic}\\
    Let $T$ be the theory of $(\mathbb{Z},0,1,+,-,<,P_n)$ where $ P_n(a)$ holds if and only if $n$ divides $a$.
    \begin{remark}\label{RmkPresbArith}
        Let $\mathcal{M}$ be a model of $T$, and let $\bar{a}=a_0,\ldots, a_{k-1}$ be a finite tuple of elements in an elementary extension $\mathcal{N}$ of $\mathcal{M}$. Then, by quantifier elimination: 
        \[\tp(\bar{a}) \cup \bigcup_{z_0,\ldots, z_{k-1}\in \mathbb{Z}} \tp(\sum_{i<k}z_ia_{i}/M)  \vdash \tp(\bar{a}/M).\]
        It follows that for all $n \in \mathbb{N}$, \[T_1(\mathcal{M},\mathcal{N}) \Rightarrow T_n(\mathcal{M},\mathcal{N}),\]
        and
        \[T_1^u(\mathcal{M},\mathcal{N}) \Rightarrow T_n^u(\mathcal{M},\mathcal{N}).\]
    \end{remark}
    
    It is also clear that all types over $(\mathbb{Z},0,+,-,<)$ are uniformly definable. Indeed, any 1-type $\tp(a/\mathbb{Z})$ where $a$ is an element of an elementary extension $\mathcal{Z}$ of $\mathbb{Z}$, is determined by the class modulo $n$ of $a$ (for every $n$) and by whether $a<\mathbb{Z}$ or $a>\mathbb{Z}$. To summarise, one has $\mathbb{Z} \preceq^{ust} \mathcal{Z}$ for every elementary extension $\mathcal{Z}$ of $(\mathbb{Z},0,+,<)$.

    We conclude this small paragraph with a remark on ordered abelian groups:
    \begin{remark} \label{RemarkZROnlyGrupStablyEmbeddedInAnyElementaryExtension}
        The only ordered abelian groups $\mathcal{Z}$ which satisfy $T_n(\mathcal{Z})$ for every $n$ are the trivial group, $\mathbb{Z}$ and $\mathbb{R}$.
    \end{remark}
    \begin{proof}
    Consider a non-trivial ordered abelian group  $\mathcal{Z}$ which is stably embedded in all elementary extensions $\mathcal{Z}' \succeq \mathcal{Z}$. It needs in particular to be archimedian as otherwise an elementary extension will realise an irrational cut of the form $\mathcal{Z}_{<\omega\cdot a}/ \mathcal{Z}_{>\omega\cdot a} $ where 
    \[\mathcal{Z}_{<\omega\cdot a} = \{z\in\mathcal{Z} \ \vert \ \text{there is }n<\omega \text{ such that }z<n\cdot a\},\]
    and 
    \[\mathcal{Z}_{>\omega\cdot a} = \{z\in\mathcal{Z} \ \vert \ \text{for all }n<\omega,\ z>n\cdot a\}\]
    where $a$ is an element of $\mathcal{Z}$ such that the above sets are not empty.
    In other words, $\mathcal{Z}$ is isomorphic as an ordered abelian group to an additive subgroup of $ \mathbb{R}$. We assume that it is a subgroup of $\mathbb{R}$. If $\mathcal{Z}$ is discrete, it is isomorphic to $\mathbb{Z}$. Assume it is not discrete, and so dense in $\mathbb{R}$. If there is $a\in \mathbb{R} \setminus \mathcal{Z}$, this element realises an irrational cut over $\mathcal{Z}$, and thus there would be an elementary extension $\mathcal{Z}'$ realising this irrational cut, which is a contradiction. This shows that $\mathcal{Z}=\mathbb{R}$.
    \end{proof}

\subsection{Preliminaries on algebraic structures}\label{SectionModelTheoryAlgebraicStructures}
\subsubsection{Valued fields}


\textbf{Notation and languages}\\
A valued field will be typically denoted by $\mathcal{K}=(K,\Gamma,k,\val)$ where $K$ is the field (main sort), $\Gamma$ the value group and $k$ the residue field. The valuation is denoted by $\val$, the maximal ideal $\mathfrak{m}$ and the valuation ring $\mathcal{O}$.  Thanks to Remark \ref{RmkTeq}, we will be able to work indifferently in different (many-sorted) languages. In the next section, we will mainly work in the many-sorted languages $\mathrm{L}_{\RV}$ and $\mathrm{L}_{\RV_{<\omega}}$, which involves the leading term structures. For the definition of the leading term structures $\RV$ and $\RV_{<\omega}$, we refer the reader to \cite{Tou20a}. They can also use \cite{Fle11} as a reference.  We recall first two traditional languages of valued fields.\\
    
    \begin{itemize}
        \item $\mathrm{L}_{\text{div}}=\{K,0,1,+,\cdot,\mid\} $, where $\mid$ is a binary relation symbol, interpreted by the division: 
        \[\text{for }a,b \in K, a \mid b \text{ if and only if } \val(a) \leq \val(b).\]
        \item $\mathrm{L}_{\Gamma,k}= \{K,0,1,+,\cdot\} \cup\{k,0,1,+,\cdot\} \cup \{\Gamma,0,\infty,+,<\} \cup \{\val: K \rightarrow \Gamma, \Res: K^2 \rightarrow k\}$. \nomenclature[L]{$\mathrm{L}_{\Gamma,k}$}{}
    \end{itemize}
    where $\Res: K^2 \rightarrow k$ is the two-place residue map, interpreted as follows:
    \[ \Res(a,b) = 
        \begin{cases}
            \res(a/b) \text{ if } \val(a)\geq \val(b) \neq \infty,\\
            0 \text{ otherwise}.
        \end{cases}\]

   For $n \in \mathbb{N}^\star$, we denote by $\mathcal{O}_n:= \mathcal{O}/ \mathfrak{m}^n$ the valuation ring of order $n$. Recall that \textit{an angular component of order} $n$ is a homomorphism $\ac_n:  K^\star \rightarrow \mathcal{O}_n^\times$ such that for all $u\in \mathcal{O}^\times$, $\ac_n(u)=res_n(u):= u + \mathfrak{m}^n$. A system of angular component maps $(ac_n)_{n<\omega}$ is said to be \textit{compatible} if for all $n$, $\pi_n \circ \ac_{n+1}=\ac_n$\index{Angular component! of order $n$}\nomenclature[]{$\ac_n$}{Angular component of order $n$} where $\pi_n:\mathcal{O}_{n+1}\rightarrow \mathcal{O}_n$ is the natural projection. The convention is to contract $\ac_1$ to $\ac$.\\ 
    A section $s:\Gamma \rightarrow K^\star$ of the valuation gives immediately a compatible system of angular components (defined as $\ac_n:= a\in K^\star \mapsto \res_n(a/s(v(a)))$). As $\mathcal{O}^\times$ is a pure subgroup of $K^\star$, such a section exists when $\mathcal{K}$ is $\aleph_1$-saturated (see Fact \ref{FactSectionPureSubgroupAleph1Saturated}). As always, we assume that $\mathcal{K}$ is sufficiently saturated and we fix a compatible sequence $(\ac_n)_n$ of angular components.  
    We have the following diagram:  
    \[\xymatrix{
        1 \ar@{->}[r] & \mathcal{O}^\times \ar@{->}[r]\ar@{->}[d]_{\res_n}& K^\star \ar@{->}[r]_\val\ar@{->}[d]_{\rv_n}\ar@/_1.0pc/@{->}[dl]^{\ac_n} & \Gamma\ar@{->}[r] \ar@{=}[d]\ar@/_1.0pc/@{-->}[l]^s &  0 \\
        1 \ar@{->}[r]& \mathcal{O}^\times /(1+\mathfrak{m}^n) \simeq \mathcal{O}_n^\times\ar@{->}[r]  & \RV_n^{\star} \ar@{->}[r]^{\val_{\rv_n}} & \Gamma \ar@{->}[r]  & 0 }\]

    Any theory $T$ of Henselian valued fields in the language above admits a natural expansion -- denoted by $T_{\ac}$ -- in the language 
    \begin{itemize}
        \item $\mathrm{L}_{\Gamma,k,\ac}:=\mathrm{L}_{\Gamma,k} \cup \{\ac: K \rightarrow k\},$
    \end{itemize} 
    by adding the axiom saying that $\ac$ is an angular component.
    Similarly, we denote by $T_{ac_{<\omega}}$ the extension of $T$ to the language 
    \begin{itemize}
        \item $\mathrm{L}_{\Gamma,k,\ac_{<\omega}}:=\mathrm{L}_{\Gamma,k} \cup \{\mathcal{O}_n, \ac_n: K \rightarrow \mathcal{O}_n, \ n\in \mathbb{N}\},$
    \end{itemize}
    where $\ac_n$ are interpreted as compatible angular components of degree $n$.

    The \textit{leading term language} $\mathrm{L}_{\RV}$ is the multisorted language \begin{itemize}
        \item $\mathrm{L}_{\RV}:=\{K,0,1,+,\cdot\} \cup \{\RV,\mathbf{0},\textbf{1},\oplus ,\cdot\} \cup \{ \rv:K \rightarrow \RV\},$
    \end{itemize}
where $\RV$ will be interpreted as the first leading term structure, and $\rv: K \rightarrow \RV$ by the natural projection map. The symbol $\oplus$ is a ternary predicate, interpreted as follows:
\begin{align*}
 \oplus (\mathbf{x},\mathbf{y},\mathbf{z}) & \equiv & \exists a,b \in K \ \nomenclature[]{$\oplus(\mathbf{x},\mathbf{y},\mathbf{z})$}{} \rv(a)=\mathbf{x} \wedge \rv(b)=\mathbf{y} \wedge \rv(a+b)=\mathbf{z}
 \end{align*}

In the case of unramified mixed characteristic Henselian valued fields, we will have to consider the following language: \begin{itemize}
    \item $\mathrm{L}_{\RV_{<\omega}} :=  \{K,0,1,+,\cdot\} \cup \{\RV_{n},\textbf{0},\textbf{1},\cdot\}_{n\in \mathbb{N}} \cup \{\oplus_{l,m,n} \subset \RV_l\times \RV_m \times \RV_n\}_{n \leq l, m \in \mathbb{N}}  \cup\{\rv_{n}: \mathcal{K} \rightarrow \RV_{n}\}_{n\in \mathbb{N}} \cup \{\rv_{m \rightarrow k}: \RV_n \rightarrow \RV_m \}_{m \leq n}$
\end{itemize}
where $\RV_n$ is interpreted as the $n^{\text{th}}$ leading term structure, $\rv_n: K \rightarrow \RV_n$ by the natural projection map and $\oplus_{l,m,n}$ by the following ternary predicate:
\begin{align*}
 \oplus_{l,m,n} (\mathbf{x},\mathbf{y},\mathbf{z}) & \equiv & \exists a,b \in K \ \nomenclature[]{$\oplus_{\delta_1,\delta_2,\delta_3}(\mathbf{x},\mathbf{y},\mathbf{z})$}{} \rv_{l}(a)=\mathbf{x} \wedge \rv_{m}(b)=\mathbf{y} \wedge \rv_{n}(a+b)=\mathbf{z}.
 \end{align*}

\paragraph{ \textbf{Benign theory of Henselian valued fields}}\label{Preliminaries Benign theory of Henselian valued fields}
    Later in this text, we prove transfer principles for some rather nice Henselian valued fields, that we called here `benign' (see Definition \ref{DefBenign} below). The goal of this paragraph is to discuss essential properties that these benign Henselian valued fields share. We will emphasise model theoretical properties, and briefly recall from which algebraic properties they can be deduced. The idea is to implicitly work axiomatically by listing required properties needed for proving \ref{TheoremReductionGammakElementaryPairs} and other results.   \\

    Let $T$ be a (possibly incomplete) theory of Henselian valued fields. We define some properties. Recall that an extension $\mathcal{K}'=(K',\RV',\Gamma',k')$ of $\mathcal{K}=(K,\RV,\Gamma,k)$ is said to be \textit{immediate} if $\Gamma'=\Gamma$ and $k'=k$. We denote the following hypothesis:
    \begin{align}
    \tag{Im} \begin{array}{c}
          \text{The set of models of $T$ is closed }
         \text{under maximal immediate extensions.}
    \end{array} \nomenclature[P]{(Im)}{}
    \end{align}
    
    It is easy to see that extensions preserving the $\RV$-sort are exactly immediate extensions, as the following commutative diagram:
    
    \[ \xymatrix{
    &&{\RV'}^{\star}\ar[dr]&&\\
1 \ar[r] & k^{\times}\ar[ur]\ar[dr] && \Gamma\ar[r] &0, \\
    &&\RV^{\star}\ar@{^{(}->}[uu]\ar[ur]&&} \]
     implies $\RV=\RV'$.\\
    
    If (Im) is satisfied, maximal immediate extensions are natural examples where one can apply the Ax-Kochen-Ershov principle. We name two such properties: 
    
    \begin{align}
    \tag*{$(\text{AKE})_{\Gamma,k}$} \text{ for } \mathcal{K},\mathcal{K}' \models T, \mathcal{K}\subseteq \mathcal{K}', \text{ we have } \ \mathcal{K} \preceq \mathcal{K}' \Leftrightarrow k \preceq k' \text{ and } \Gamma \preceq \Gamma'.
    \end{align} \nomenclature[P]{$(\text{AKE})_{\Gamma,k}$,$(\text{AKE})_{\RV}$}{}
    \begin{align} \tag*{$(\text{AKE})_{\RV}$}\text{ for } \mathcal{K},\mathcal{K}' \models T,  \mathcal{K}\subseteq \mathcal{K}', \text{ we have } \mathcal{K} \preceq\mathcal{K}' \Leftrightarrow \RV \preceq \RV'.
    \end{align}
    These properties $(\text{AKE})_{\RV}$ and $(\text{AKE})_{\Gamma,k}$ are two different points of view, and often come together. We also denote the following properties:
    \begin{align}
        \tag*{$(\text{SE})_{\Gamma,k}$}\text{$\Gamma$ and $k$ are pure, stably embedded and orthogonal.} 
   \end{align}
    \begin{align}
        \tag*{$(\text{SE})_{\RV}$} \text{ $\RV$ is pure and stably embedded.}
    \end{align}
    
    They are consequences of relative quantifier elimination:
    \begin{equation}
    \tag{$(\text{EQ})_{\Gamma,k,ac}$} \begin{array}{l}
        \text{ $T_{\ac}$ has quantifier elimination (resplendently)}  \\ 
          \text{relatively  to $\Gamma$ and $k$ in the language $\mathrm{L}_{\Gamma,k,\ac}$. }
          \end{array}
    \end{equation}
        \begin{equation}
    \tag{$(\text{EQ})_{\RV}$} \begin{array}{l}
        \text{ $T$ has quantifier elimination (resplendently)}  \\ 
          \text{relatively  to $\RV$ in the language $\mathrm{L}_{\RV}$. }
          \end{array}
    \end{equation}
     \nomenclature[P]{$(\text{EQ})_{\Gamma,k,ac}$,$(\text{EQ})_{\RV}$,$(\text{SE})_{\Gamma,k}$,$(\text{SE})_{\RV}$}{}
    
    Notice that according to the terminology in \cite{Rid17}, $\{\Gamma\},\{k\}$ and $\{\RV\}$ are closed sets of sorts. Then resplendency automatically follows from relative quantifier elimination ( \cite[Proposition A.9]{Rid17}). Here is a well known fact followed by an easy observation.
    \begin{fact}
            $(\text{EQ})_{\Gamma,k,\ac}$ implies $(\text{AKE})_{\Gamma,k}$ and $(\text{SE})_{\Gamma,k}$.\\
            $(\text{EQ})_{\RV}$ implies $(\text{AKE})_{\RV}$ and $(\text{SE})_{\RV}$.
    \end{fact}
    
    \begin{observation}
        \begin{enumerate}
            \item $(\text{AKE})_{\Gamma,k}$ implies $(\text{AKE})_{\RV}$.
            \item $(\text{EQ})_{\RV}$ implies $(\text{EQ})_{\Gamma,k,ac}$.
        \end{enumerate}
    \end{observation}

    \begin{proof}
        (1) This is immediate, as the value group and residue field are interpretable in the $\RV$-structure: for any models $\mathcal{M},\mathcal{N} \models T$, $ \RV_M \preceq \RV_N$ implies that $ k_M \preceq k_N \text{ and } \Gamma_M \preceq \Gamma_N$. \\
        (2) See for example \cite{Tou20a}.
    \end{proof}

    We gather under the name of \emph{benign} some nice theories of Henselian valued fields which satisfy these properties.
    \begin{definition}
        A valued field of equicharacteristic $p>0$ is said Kaplansky if the value group is $p$-divisible, the residue field is perfect and does not admit any finite separable extensions of degree divisible by $p$.
    \end{definition}

    \begin{definition}\label{DefBenign} \index{Valued field! benign Henselian}
        Any $\{\Gamma\}$-$\{k\}$-enrichment of one of the following theories of Henselian valued fields is called \emph{benign}:

    \begin{enumerate}
        \item Henselian valued fields of characteristic $(0,0)$,
        \item algebraically closed valued fields,
        \item algebraically maximal Kaplansky Henselian valued fields.
    \end{enumerate}
        A model of a benign theory will be called a \emph{benign Henselian valued field}.
    \end{definition}

    \begin{fact}\label{FactBenignTheories}
    Benign theories satisfy $(\text{Im})$ and $(\text{EQ})_{\RV}$. 
    \end{fact}
    By the discussion above, they also satisfy $(\text{EQ})_{\Gamma,k,ac}$, $(\text{AKE})_{\Gamma,k}$, $(\text{AKE})_{\RV}$,$(\text{SE})_{\Gamma,k}$ and $(\text{SE})_{\RV}$.
    
    It is clear that the set of models of a benign theory is closed under maximal immediate extensions. The reader can refer to \cite{Fle11},\cite{HKR18} and \cite{HH19} for proofs of relative quantifier elimination. Notice that it also holds for any $\{\Gamma\}$-$\{k\}$-enrichment, as it is a particular case of $\{\RV\}$-enrichment, and as the sorts $\Gamma, k$ and $\RV$ are closed (and by \cite[Proposition A.9]{Rid17}).

We will complete our study with some transfer principle for unramified mixed characteristic Henselian valued fields with perfect residue field. As it requires further techniques, it needs to be treated independently.

\paragraph{\textbf{Unramified mixed characteristic Henselian valued fields}}
\label{Preliminaries Unramified mixed characteristic Henselian valued fields}

    We give a short overview on unramified valued fields, by presenting the similarities with benign valued fields.
    The (partial) theory of Henselian valued fields of unramified mixed characteristic does not satisfy either $(\text{EQ})_{\Gamma,k,ac}$ or $(\text{EQ})_{\RV}$. We indeed need to get `information' modulo $\mathfrak{m}^{n}$ in a quantifier-free way. 
    
    Let us cite a result in \cite{Fle11}:
    \begin{fact}\cite[Proposition 4.3]{Fle11}\label{FactRelQERV}
        Let $T$ be the theory of unramified mixed characteristic Henselian valued fields in the language $\mathrm{L}_{\RV_{<\omega}}$. Then $T$ has quantifier elimination (resplendently) relatively to $\{\RV_n\}_{n\in \mathbb{N}}$. 
    \end{fact}
    
    Resplendency also comes from the fact that set of sorts $\{\RV_{n}\}$ is closed. Notice that it holds in fact for any Henselian valued field of characteristic 0. This result was already proved in \cite{Bas91}. Again, an important consequence is that the multisorted substructure  $\left\{(\RV_{n})_{n<\omega},(\oplus_{n,m,l})_{n<l,m}, (\rv_{n\rightarrow m})_{m<n<\omega} \right\}$ is stably embedded and pure with control of parameters.

    Now, let us discuss more specifically the case of perfect residue field. We denote now by $T$ the theory of unramified mixed characteristic Henselian valued fields with perfect residue field. 
    The following proposition is well known and has been used for example in \cite[Corollary 5.2]{Bel99}. It states how the structure $\RV_n$ and the truncated Witt vectors $W_n$ are related.
    \begin{proposition}[{\cite[Proposition 1.69]{Tou20a}}]\label{KerValRV_n}
    Assume the residue field $k$ to be perfect.
         \begin{enumerate}
             \item The residue ring $\mathcal{O}_n$ of order $n$  is isomorphic to $W_n(k)$, the ring of truncated Witt vectors of length $n$ .
             \item The kernel of the valuation $\val:  \RV_n^\star \rightarrow \Gamma$ is given by $\mathcal{O}^\times/(1+\mathfrak{m}^n) \simeq (\mathcal{O}/\mathfrak{m}^n)^\times$. It is isomorphic to $W_n(k)^\times$, the set of invertible elements of $W_n(k)$.
         \end{enumerate}

    \end{proposition}

    \begin{fact}[Bélair \cite{Bel99}]\label{FactBelair}
     The theory $T_{ac_{<\omega}}$ of Henselian mixed characteristic valued fields with perfect residue field and with angular components eliminates field-sorted quantifiers in the language $\mathrm{L}_{\Gamma,k,\ac_{<\omega}}$\footnote{The Ax-Kochen-Ershov property and relative quantifier elimination for Henselian unramified mixed characteristic valued fields (with possibly imperfect residue field) has been proved in \cite{AJ19}.}.
    \end{fact}
    Notice that in \cite{Bel99}, Bélair doesn't assume that the residue field $k$ is perfect, but it is indeed necessary in order to identify the ring $\mathcal{O}_n:=\mathcal{O}/\mathfrak{m}^n$ with the truncated Witt vectors $W_n(k)$.
    This implies as well that the residue field $k$ and the value group $\Gamma$ are pure sorts, and are orthogonal. This can be seen by analysing field-sorted-quantifier-free formulas, and by noticing that $\mathcal{O}_n \simeq W_n(k)$ is interpretable in $k$ (see e.g. \cite[Corollary 1.64]{Tou20a}).

    By analogy with the previous paragraph, we name the following properties: 
    \begin{align}
    \tag*{$(\text{EQ})_{\Gamma,k,ac_{<\omega}}$} \text{ $T_{ac_{<\omega}}$ eliminates $K$-sorted quantifiers in the language $\mathrm{L}_{\ac_{<\omega}}$}.\\
    \tag*{$(\text{EQ})_{\RV_{<\omega}}$} \text{ $T$ has quantifier elimination (resplendently) relatively to $\RV_{<\omega}$}.
    \end{align}
    \nomenclature[P]{$(\text{EQ})_{\Gamma,k,ac_{<\omega}}$,$(\text{EQ})_{\RV_{<\omega}}$}{}
    Again, notice that $\RV_{<\omega}=\bigcup_{n<\omega}\RV_n$ is a closed set of sorts. As before, we get:
   \begin{fact}
            $(\text{EQ})_{\RV_{<\omega}}$ implies $(\text{AKE})_{\RV_{<\omega}}$.
    \end{fact}
    
    We also need to adapt the axiom $(\text{Im})$, as it is probably safer to look for a stronger property. Indeed, one can ask the following:
  \begin{question}\label{QImmediate}
        When do we have that, for every $n$, $\RV_n=\RV_n'$ for all immediate extensions $\mathcal{K}'/\mathcal{K}$?
    \end{question}
    
     In general, an immediate extension of a mixed characteristic field can have a larger $\RV_n$-sort. For instance, let us look at the field of rational functions $K=\mathbb{Q}(X)$ and the field of formal power series $K'=\mathbb{Q}((X))$. 
    We consider the valuation ring $\mathcal{O}'$ on $K'$ defined by
    \[\mathcal{O}':=\{\sum_{i\geq n}a_iX^i \ \vert n\in \mathbb{Z}, \ \val_p(a_i) \geq 0 \text{ for all $i$ and } \val_p{a_i}=0 \text{ only if $i\geq 0$}\}.\]
     The value group can be identified with $(\mathbb{Z}\times \mathbb{Z}, (0,0),+ , <_{lex})$ endowed with the lexicographic order $<_{lex}$ and where $\val(p)=(0,1)> \val(X)=(0,1)$. We consider the restriction of the valuation to $K$.  Then the extension $K'/K$ is immediate. One sees that $\RV_{(1,0)}(K')$ is uncountable (isomorphic as abelian group to $\mathbb{F}_p((X))\times \Gamma$), and that $\RV_{(1,0)}(K)$ is countable (isomorphic as an abelian group to $\mathbb{F}_p(X)\times \Gamma$). 

    An extension $\mathcal{K}' /\mathcal{K}$ satisfying $\RV_{<\omega}' =\RV_{<\omega}$ will be called $\RV_{<\omega}$\textit{-immediate}. Let us give a name to the condition saying that the previous question has a positive answer:
    \begin{align*}
    \tag*{$(\text{Im})_{\RV_{<\omega}}$}
    \begin{array}{l}
         \text{The set of models of $T$ is closed under maximal} \\
           \text{ immediate extensions and immediate extensions of}\\
        \text{ models of $T$ are $\RV_{<\omega}$-immediate}.
    \end{array}
    \end{align*} \nomenclature[P]{$(\text{Im})_\RV_{<\omega}$}{}
    This is satisfied by unramified mixed characteristic valued fields with perfect residue field. Indeed, the same argument as before proves that $\RV_n=\RV_n'$ for all $n \in \mathbb{N}$, as one has the commutative diagram: 
 
     \[ \xymatrix{
    &&{\RV_n'}^\star\ar[dr]&&\\
1 \ar[r] & W_n(k)^{\times}\ar[ur]\ar[dr] && \Gamma\ar[r] &0, \\
    &&\RV_n^\star\ar@{^{(}->}[uu]\ar[ur]&&} \]
    
    where $W_n(k)$ is the ring of Witt vectors of order $n$ over $k$.

    To sum up, we have:

     \begin{fact}\label{FactUnramifiedMixedCharacteristicValuedfields}
    The theory of unramified mixed characteristic Henselian valued fields with perfect residue field satisfies $(\text{Im})_{\RV_{<\omega}}$, $(\text{EQ})_{\RV_{<\omega}}$, $(\text{EQ})_{\Gamma,k,ac_{<\omega}}$, 
    $(\text{SE})_{\Gamma,k}$ and $(\text{SE})_{\RV_{<\omega}}$. 
    \end{fact}

\subsubsection{Pure short exact sequences of abelian groups}\label{SubsectionPreliminariesAbelianGroups}
We conclude these preliminaries with some facts on abelian groups. They will be use for proving a reduction principle in short exact sequence of abelian group, and after enrichment, for proving our reduction principle in Henselian valued fields.

    \begin{definition}
        Let $B$ a group and $A$ a subgroup. We say that $A$ is a \emph{pure subgroup} \index{Pure!subgroup} of $B$ if for all $a$ in $A$, $n\in \mathbb{N}$, $a$ is $n$-divisible in $B$ if and only if $a$ is $n$-divisible in $A$.
    \end{definition}
    We recall the following fundamental fact:
    \begin{fact}\label{FactSectionPureSubgroupAleph1Saturated}
        Let $\mathcal{M}$ be an $\aleph_1$-saturated structure, and let $A,B$ be two definable abelian groups, and assume that $A$ is a pure subgroup of $B$. Then the exact sequence of abelian groups $0\rightarrow A \rightarrow B \rightarrow B/A \rightarrow 0$
        splits: there is a group homomorphism $\alpha:B \rightarrow A$ such that $\restriction{\alpha}{A}$ is the identity on $A$. In such case, $B$ is isomorphic as a group to $A \times B/A$.
    \end{fact}
    More precisely, it is an immediate corollary of a more general statement on \emph{pure-injectivity}. See \cite[Theorem 20 p.171]{Che76}.
    
    Assume that we have a pure short exact sequence of abelian groups
    \[ \xymatrix{0 \ar[r] & A \ar[r]^{\iota}& B \ar[r]^{\nu} & C\ar[r] &0}. \]
    (meaning that $\iota(A)$ is a pure subgroup of $B$)\index{Pure! short exact sequence of abelian groups}. We treat it as a three-sorted structure $(A,B,C,\iota,\nu)$, with a group structure for all sorts. In fact, in our main applications, we will consider such a sequence with more structure on $A$ and $C$. Let us explicitly state all results resplendently, by working in an enriched language. So, let $\mathcal{M}=(A,B,C,\iota,\nu,\ldots)$ be an $\{A\}$-enrichment of a $\{C\}$-enrichment  (for short: an $\{A\}$-$\{C\}$-enrichment) of the exact sequence in a language that we will denote by $\mathrm{L}$, and we denote its theory by $T$. We will always assume that $\mathcal{M}$ is sufficiently saturated ($\aleph_1$ saturated will be enough).\\
    Hypothesis of purity implies the exactness of the following sequences for $n\in \mathbb{N}$:
    \[\xymatrix{0 \ar[r] & A/nA \ar[r]^{\iota_n}& B/nB \ar[r]^{\nu_n} & C/nC\ar[r] &0}.\]
    One has indeed that 
    \[\frac{A+nB}{nB} \simeq \frac{A}{A\cap nB} = \frac{A}{nA}.\]
    
    We consider for $n\geq 0$ the following maps: \\
    \begin{itemize}
        \item the natural projections $\pi_n:A \rightarrow A/nA$,
        \item the map  \[\begin{array}{ccccc}
\rho_n  : & B & \to & A/nA \\
  & b & \mapsto & 
        \begin{cases}
            0_{A/nA} \ \text{ if }b\notin \nu^{-1}(nC) \\
            \iota_n^{-1}(b+nB) \text{ otherwise, }
        \end{cases} 
        \end{array}\]
        \end{itemize}

    where $0_{A/nA}$ is the zero element of $A/nA$ (often denoted by $0$).
    Then let us consider the language 
    \[\mathrm{L}_q= \mathrm{L} \cup \{A/nA,\pi_n,\rho_n\}_{n\geq 0},\] \nomenclature[L]{$\mathrm{L}_q$}{}
    and let $T_q$ be the natural extension of the theory $T$. 
    By $<A>$, we denote the set of sorts containing $A$, $A/nA$ and the new sorts possibly coming from the $A$-enrichment. Similarly, let $<C>$ be the set of sorts containing $C$ and the new sorts possibly coming from the $C$-enrichment. By $A$-sort and $C$-sort, we will abusively refer to $<A>$ and $<C>$ respectively, and similarly for $A$-formulas and $C$-formulas.   Aschenbrenner, Chernikov, Gehret and Ziegler prove the following result:
    
    \begin{fact}[{\cite[Theorem 4.2]{ACGZ20}}]\label{FactACGZ}
        The theory $T_q$ (resplendently) eliminates $B$-sorted quantifiers.
        

    More precisely, all $\mathrm{L}_q$-formulas $\phi(x)$ with a tuple of variables  $x\in B^{\vert x\vert}$ are equivalent to boolean combinations of formulas of the form:
    \begin{enumerate}
        \item $\phi_C(\nu(t_0(x)),\ldots,\nu(t_{s-1}(x)))$ where $t_i(x)$'s are terms in the group language, and $\phi_C$ is a $C$-formula,
        \item $\phi_{A}(\rho_{n_0}(t_0(x)),\ldots,\rho_{n_{s-1}}(t_{s-1}(x)))$ where the $t_i(x)$'s are terms in the group language, where \newline $s,n_0,n_1,\ldots,n_{s-1}\in \mathbb{N}$, and where $\phi_{A}$ is an $A$-formula.
    \end{enumerate}
    In particular there is no occurrence of the symbol $\iota$.
    
    \end{fact}
    
    In particular, notice that the formula $t(x)=0$ is equivalent to $\nu(t(x))=0 \wedge \rho_0(t(x))=0$ and $\exists y \ ny=t(x)$ is equivalent to $\exists y_C \ ny_C=\nu(t(x)) \wedge \rho_n(t(x))=0$.
    
    We have:
	    \begin{corollary}\label{CoroPureOrth}	
	        In the theory $T_q$, $<A>$ and $<C>$ are stably embedded, pure with control of parameters and orthogonal to each other.
	    \end{corollary}
	The reader can refer to \cite{Tou20a} for definitions and a proof of this corollary.

\section{Stably embedded sub-valued fields}\label{SectionStablyEmbeddedPairs}
    We prove our transfer principle for benign valued fields (Theorem \ref{TheoremReductionGammakNonElementaryPairs} and Theorem \ref{TheoremReductionGammakElementaryPairs}) in Subsection \ref{SectionReductionStableEmbeddednessBenignValuedFields}.
   Our transfer principle for unramified mixed characteristic Henselian valued fields with perfect residue field (Theorem \ref{TheoremReductionGammakElementaryPairsMixeCharacteristic}) is proved in Subsection \ref{SectionReductionStableEmbeddednessUnramifiedMixedCharacteristic}, using some additional arguments.
   Finally, in Subsection \ref{SectionPairs}, we discuss the elementarity (in the language of pairs) of the class of elementary pairs $\mathcal{K} \prec \mathcal{L}$ 
    where all types over $\mathcal{K}$ realised in $\mathcal{L}$ are definable (equivalently where $\mathcal{K}$ is stably embedded in $\mathcal{L}$).

    \subsection{Benign Henselian valued fields}\label{SectionReductionStableEmbeddednessBenignValuedFields}
        
    Consider a benign theory $T$ of Henselian valued fields. We want to discuss when a valued field $\mathcal{K}$ is stably embedded (resp.\ uniformly stably embedded) in an extension $\mathcal{L}$ which is a model of $T$. We need first to define the following:
    \begin{definition}\label{DefinitionSeparated}
    An extension of valued fields $\mathcal{L}/ \mathcal{K}$ is said separated\index{Separatedness of valued fields extension} if for any finite-dimensional $K$-vector subspace $V$ of $L$, there is a $K$-basis $\{c_0, \ldots,c_{n-1}\}$ of $V$ such that for any $( a_0, \ldots, a_{n-1})\in K^{n}$,
    \[v(\sum\limits_{i<n}a_ic_i) = \min\limits_{i<n}(v(a_ic_i)).\]  
    Equivalently, this means that for any $\lbrace a_0, \ldots, a_{n-1}\rbrace \in K^{n}$,
    \[\rv(\sum_{i<n} a_ic_i)= \bigoplus_{i<n} \rv(a_i)\rv(c_i).\]
    Such a basis $\{c_0,\ldots, c_{n-1}\}$ is called a \emph{separating basis of $V$ over $\mathcal{K}$}.
    \end{definition}
    As we will see, it is a necessary condition for an elementary extension $\mathcal{L}/\mathcal{K}$ to be separated in order to be stably embedded. This property has been intensively studied (see \cite{Del88} and \cite{Bau80} for more details).\\
    
    Let us state the theorems of Cubides-Delon and Cubides-Ye:
    \begin{theorem}[{\cite[Theorem 1.9]{CD16}}]\label{ThmCubidesDelon}
        Consider $\mathcal{K} \preceq \mathcal{L}$ be two algebraically closed valued fields. The following are equivalent:
            \begin{enumerate}
                \item $\mathcal{K}$ is stably embedded (resp.\ uniformly stably embedded) in $\mathcal{L}$,
                \item the extension $\mathcal{L}/\mathcal{K}$ is separated and $\Gamma_K$ is stably embedded (resp.\ uniformly stably embedded) in $\Gamma_L$.
            \end{enumerate}
    \end{theorem}

    \begin{theorem}[{\cite[Theorem 5.2.3.]{CY19}}]\label{ThmCubidesYe}
        Consider $\mathcal{K} \preceq \mathcal{L}$ be two real closed valued fields, or two p-adically closed valued fields. The following are equivalent:
            \begin{enumerate}
                \item $\mathcal{K}$ is stably embedded (resp.\ uniformly stably embedded) in $\mathcal{L}$,
                \item the extension $\mathcal{L}/\mathcal{K}$ is separated, $k_K$ is stably embedded (resp.\ uniformly stably embedded) in $k_L$ and $\Gamma_K$ is stably embedded (resp.\ uniformly stably embedded) in $\Gamma_L$.
            \end{enumerate}
    \end{theorem}
    
    Only the non-uniform case is stated in {\cite[Theorem 1.9]{CD16}}, but the proof goes through for the uniform case also. In \cite{CY19}, the uniform case can be deduced from the proof or by \cite[Theorem 6.0.3]{CY19}. We will generalise these theorems to extensions $\mathcal{L}/ \mathcal{K}$ of benign theories $T$ of Henselian valued fields (Paragraph \ref{Preliminaries Benign theory of Henselian valued fields}). In fact, we will not assume that $\mathcal{K}$ and $\mathcal{L}$ share the same complete theory. We proceed in two steps. For a separated extension of benign Henselian valued fields $\mathcal{L}/ \mathcal{K}$, we characterise $K \subseteq^{st} \mathcal{L}$ (resp.\ $K \subseteq^{ust} \mathcal{L}$) by such property of the $\RV$-sorts (Paragraph \ref{SectionRV}), and later by such properties of the value groups and the residue fields (Paragraph \ref{SectionGammaK}). 
    
    \subsubsection{Separatedness as a necessary condition}\label{SectionSeparatednessNecessaryCondition}
    We are going to prove that elementary submodels of benign valued fields are stably embedded only if the extension is separated (Proposition \ref{PropRV2}). This is a generalisation of $(1 \Rightarrow 2)$ in \cite[Theorem 1.9]{CD16}. Notice that our proof requires that the extension is elementary. In the next paragraphs, it will no longer be assumed.
    \begin{proposition}\label{PropRV2}
     Let $\mathcal{L}/ \mathcal{K}$ be an elementary extension of valued fields, with $\Th(\mathcal{L})$ a completion of a benign theory of Henselian valued fields. If $\mathcal{K}$ is stably embedded in $\mathcal{L}$, then $\mathcal{L}/ \mathcal{K}$ is a separated extension of valued fields.
    \end{proposition}
    
    \begin{remark*}
        In fact, the proof below (more specifically the proof of Corollary \ref{CorollaryDefinablyCompleteSubVectorSpace}) does not require relative quantifier elimination, but only the properties $\text{(Im)}$ and  $(\text{AKE})_{\RV}$. 
    \end{remark*}
    \begin{proof}
    We start by defining the notion of valued vector spaces. 
    
    \begin{definition}
    A valued $\mathcal{K}$-vector space $\mathcal{V}$ is a $K$-vector space $V$ and a totally ordered set $\Gamma_V$ together with:
        \begin{itemize}
            \item a group action $+:\Gamma_K\times \Gamma_V \rightarrow \Gamma_V$ which is strictly increasing in both variables,
            \item a surjective map, called the \emph{valuation}, $\val: V\setminus \{0\} \rightarrow \Gamma_V$ such that $\val(w+v)\geq \min(\val(w),\val(v))$ and $\val(\alpha \cdot v)= \val(\alpha)+ \val(v)$ for all $v,w\in V$ and $\alpha\in K$. By convention, $\val(0)=\infty$.
        \end{itemize}
    \end{definition}
    Of course, the notion of separating basis and separated vector space extend naturally to this slightly more general setting. For $v\in V$ and $\gamma\in \Gamma_V$, we define the closed ball $B_{\geq \gamma}(v)$ by $\{v'\in V \ \vert \ \val(v-v')\geq \gamma\}$ and the open ball $B_{> \gamma}(v)$ by $\{v'\in V \ \vert \ \val(v-v')> \gamma\}$.
    We borrow the following lemma and its proof from a lecture course of Martin Hils:
    \begin{lemma}
        Assume $\mathcal{K'}=(K',\val)$ to be a maximal valued field. Let $\mathcal{V}'=(V',\val)$ be a separated finite dimensional valued $\mathcal{K}'$-vector space. Then $V'$ is spherically complete: if $(D_i)_{i\in I}$ is a family of nested balls, then the intersection $\bigcap_{i\in I}D_i$ is non-empty.
    \end{lemma}
    \begin{proof}
        We proceed by induction on $n\geq 1$. If $V'$ is of dimension $1$, we simply have that $(K',\val) \simeq (V',\val)$ as $\mathcal{K}'$-valued vector spaces. Then, we only need to recall that $\mathcal{K}'$ is pseudo-complete as a maximal valued field.
        Assume the lemma to holds for all sub-$\mathcal{K}'$-vector spaces of dimension $n$ and let $W=\{w_0,\ldots,w_n\}$ be a separating basis of  $V'$, a sub-$\mathcal{K}'$-vector space of dimension $n+1$. Let $(B_\alpha=B_{\geq \gamma_{\alpha}}(v_{\alpha}))_{\alpha<\lambda}$ be a decreasing sequence of closed balls in $\mathcal{V}'$, with $\lambda$ a limit ordinal. For $\alpha<\lambda$, write $v_{\alpha}= \sum_{i\leq n} a_{\alpha,i} w_i$ with $a_{\alpha,i} \in K'$. Applying the definition of separating basis and by taking a subsequence, we may assume that  
        $\gamma_\alpha \leq \val(v_{\alpha+1} -v_{\alpha})=\val(a_{\alpha+1,n} -a_{\alpha,n})+\val(w_n)$ for all $\alpha<\lambda$. 
        It follows that the sequence $(a_{\alpha,n})_{\alpha<\lambda}$ is pseudo-Cauchy in $\mathcal{K}'$. As $\mathcal{K}'$ is pseudo-complete, one finds a pseudo-limit $a_n$ in $K'$. We consider now the sequence $(v_{\alpha}')_{\alpha<\lambda}$ where $v_{\alpha}' = \sum_{i<n}a_{\alpha,i} w_i + a_n w_n$. One has the following:
        \begin{itemize}
            \item $v_{\alpha}' \in B_\alpha$ i.e. $B_{\geq \gamma_\alpha}(v_\alpha) = B_{\geq \gamma_\alpha}(v_\alpha')$,
            \item the set $B_\alpha':= \{\sum_{i<n} d_{\alpha,i} w_i \ \vert \ d_{\alpha,i}\in K',  \sum_{i<n} d_{\alpha,i} w_i+ a_nw_n \in B_\alpha\}$ is the (non-empty) closed ball $$B_{\geq \gamma_\alpha}(\sum_{i<n}a_{\alpha,i}w_i)$$ in $V'=<w_0,\ldots , w_{n-1}>_{K'}$.
            \item $(B_\alpha')_{\alpha<\lambda}$ is a decreasing sequence of closed balls in $V'$ (of $\mathcal{K}'$-dimension $n$). 
        \end{itemize}
        By the induction hypothesis, the sequence $(B_\alpha')_{\alpha<\lambda}$ admits a non empty intersection. Let $v' \in \cap_{\alpha<\lambda}B_{\alpha}'$. Then, one has that $v=v'+a_{n}w_n \in \cap_{\alpha<\lambda}B_{\alpha} $.

    \end{proof}

    \begin{corollary}\label{CorollaryDefinablyCompleteSubVectorSpace} 
         Assume that $\mathcal{K}$ is a stably embedded elementary submodel of $\mathcal{L}$. Any finite dimensional separated sub-$\mathcal{K}$-vector space $V$ of $\mathcal{L}$ is definably spherically complete: let $(D_i)_{i\in I}$ be a definable family of balls ($D_i$ is a closed ball defined by a parameter $i$, and $I$ is a definable set) with the finite intersection property (no finite intersection is empty). Then the intersection $\bigcap_{i\in I}D_i$ is non-empty.
    \end{corollary}
    \begin{proof}
        Let $C=\{c_0, \ldots, c_{n-1}\}$ be a separating basis of $V$. We write $\gamma_i=\val(c_i)\in \Gamma_L$ for $i<n$. One can interpret $V$ in $K^n$: elements are identified with their decomposition in the basis $C$, addition and scalar multiplication are defined as usual. Since $K$ is stably embedded in $\mathcal{L}$, the type $\tp(c_0, \ldots, c_{n-1}/K)$ is definable. It follows that $v(\sum_{i<n}a_ic_i)>v(\sum_{i<n}b_ic_i)$ holds if and only if $(a_0,\ldots,a_{n-1})$ and $(b_0,\ldots,b_{n-1})$ satisfy a certain $K$-formula. So we also interpret the valuation.
        Let $(D_i)_{i\in I}$ be a uniformly definable family of nested balls. Consider $K'$, the maximal immediate extension of $K$. By the previous lemma, the (definable) intersection $\bigcap_{i\in I'} D_i'$ has a point in $K'$ (where $I'$ is the definable set $I(K')$, $D_i'=D_i(K')$), and so $\bigcap_{i\in I} D_i$ is non-empty, as $K' \succeq K$ by the Ax-Kochen-Ershov principle.
    \end{proof}
     We prove by induction on $n$ that any sub-$\mathcal{K}$-vector space of $\mathcal{L}$ of dimension $n$ is separated. There is nothing to show for $n=1$. Let $V$ be a finite dimensional $K$-vector subspace of $L$, with a separating basis $C=\{c_0, \ldots, c_{n-1}\}$. Let $a$ be any element of $L \setminus V$. Let us show that the $K$-vector space $\tilde{V}=<V,a>_K$ generated by $V$ and $a$ is also separated.  
    It follows from Corollary \ref{CorollaryDefinablyCompleteSubVectorSpace} that $\{v(w-a) \vert \ w \in V\}$ has a maximum. Indeed, otherwise the family of balls $(B_{\geq \val(w-a)}(a))_{w\in V}$ will have an empty intersection. As $\tp(a, c_0,\ldots, c_{n-1}/\mathcal{K})$ is definable ($\mathcal{K}$ is stably embedded in $\mathcal{L}$), it is a definable family of balls, contradicting the fact that $V$ is definably spherically complete by Corollary \ref{CorollaryDefinablyCompleteSubVectorSpace}. Let $c_n=w-a$ realise this maximum. By a simple calculation, one sees that $\Tilde{C}=\{c_0, \ldots, c_n\}$ is a separating basis of $\Tilde{V}=<V,a>$. Indeed, consider any element $b\in \Tilde{V}\setminus V$ and its decomposition $b=\sum_{i\leq n}b_ic_i$ in the basis $\Tilde{C}=\{c_0, \ldots, c_n\}$.  Notice that $\frac{\sum_{i<n}b_ic_i}{b_n}$ is an element of $V$. If $\val(\sum_{i<n}b_ic_i) \geq v(b_nc_n)$, then $\val(\frac{\sum_{i<n}b_ic_i}{b_n}+c_n)= v(c_n)$ by maximality of $\val(c_n)$, which gives $\val(\sum_{i<n}b_ic_i +b_nc_n)= \val(b_nc_n)$. If $\val(\sum_{i<n}b_ic_i) < v(b_nc_n)$, then $\val(\sum_{i<n}b_ic_i +b_nc_n)= \val(\sum_{i<n}b_ic_i)$. This proves that $\Tilde{C}=\{c_0, \ldots, c_n\}$ is a separating $\mathcal{K}$-basis of $\tilde{V}$.
    \end{proof}
    
    \subsubsection{Reduction to $\RV$}\label{SectionRV}
    In this paragraph, we reduce definability over sub-valued fields to the $\RV$-sort 
    for all benign theories of Henselian valued fields. The next two propositions consist simply of an $\RV$-version 
    of the theorem of Cubides-Delon. The proof is completely similar to that of $(2 \Rightarrow 1)$ in \cite[Theorem 1.9]{CD16}. Notice that the idea of adapting the proof of Cubides-Delon in the setting of an $\RV$-sort can already be found in Rideau-Kikuchi's thesis \cite{Rid14}. 
    Let $T$ be a benign theory of Henselian valued fields. Recall that we can safely work in the language $\mathrm{L}_{\RV}$ (or in an $\RV$-enrichment of it), by Remark \ref{RmkTeq}.
    {\begin{theorem} \label{PropRV} Let $\mathcal{L}/ \mathcal{K}$ be a separated extension with $\mathcal{L} \models T$. The following are equivalent:
        \begin{enumerate} \index{$\RV$-sort}\index{Stable embeddedness}
            \item $K$ is stably embedded (resp.\ uniformly stably embedded) in $\mathcal{L}$,
            \item $\RV_K$ is stably embedded (resp.\ uniformly stably embedded) in $\RV_L$.
        \end{enumerate}
    \end{theorem}}
   The proof below use the property $(\text{EQ})_{\RV}$. 
    
    \begin{proof}
    We prove the non-uniform case, the uniform case being similar.\\  
        $(2 \Rightarrow 1)$ Let $\phi(x,a)$ be a formula with parameters in $L$, $x$ a tuple of field sorted variables. By relative quantifier elimination to $\RV$, it is equivalent to a formula of the form 
        $$\psi(\rv(P_0(x)),\ldots,\rv(P_{k-1}(x)),\mathbf{b}),$$
        where $\psi(\mathbf{x}_0,\ldots,\mathbf{x}_{k-1},\mathbf{y})$ is an $\RV$-formula, $P_i(x)$'s are polynomials with coefficients in $L$ and $\mathbf{b}\in \RV_L$. For instance, notice that the formula $P(x)=0$ for $P(X)\in L[X]$ is equivalent to $\rv(P(x))=0$.
    Consider $V$ the finite dimensional $K$-vector-space generated by the coefficients of the $P_i$'s, and let $c_0, \ldots, c_{n-1}$ be a separating basis of $V$. For $x\in K$, one has $P_i(x)= \sum_{j<n} P^j_i(x)c_j$ for some polynomials $P^j_i(x)\in K[x]$. 
    By definition of separating basis, one has for $x \in K:$ 
    \[\rv(P_i(x))= \bigoplus\limits_{j<n}\rv(P^j_i(x)c_j).\] 
    Hence, the trace in $K$ of the formula $\phi(x,a)$ is given by $\theta( (\rv(P^i_j(x)))_{i<k,j<n}, (\rv(c_j))_{j<n},\mathbf{b} )$, where $\theta$ is an $\RV$-formula. Now, using that $\RV_K$ is stably embedded in $\RV_L$, $\theta(\mathbf{x}, (\rv(c_j))_{j<n},\mathbf{b})$ can be replaced by a formula $\xi(\mathbf{x}, \mathbf{d})$ with parameters $\mathbf{d}$ in $\RV_K$, which induces the same set in $\RV_K$. At the end, we get that $\phi(K,a)$ is definable with parameters in $K$.\\
    
    $(1 \Rightarrow 2)$
    By relative quantifier elimination to $\RV$, if $\mathcal{K}$ is stably embedded in $\mathcal{L}$, so is $\RV_K$ in $\RV_L$. Indeed, consider a formula $\phi(\mathbf{x},\mathbf{a})$ with free variable $\mathbf{x}\in \RV$ and parameter $\mathbf{a}\in \RV_L$. As $K$ is stably embedded in $L$, there is an $\mathrm{L}(K)$-formula $\psi(\mathbf{x},b)$ with \textit{a priori} field-sorted parameters $b\in K$ such that $\psi(\RV_K,b)=\phi(\RV_K,\mathbf{a})$. By relative quantifier elimination to $\RV$, there is a field-sorted-quantifier-free formula $\theta(\mathbf{x},\mathbf{y})$ and field-sorted terms $t(y)$ such that 
    $\theta(\RV_K,\rv(t(b)))= \psi(\RV_K,b) = \phi(\RV_K,\rv(a)).$
    This concludes our proof.
    \end{proof}

    \begin{remark}\label{RemarkReductionToEnrichedRV}
    The proof above holds for any $\RV$-enrichment $T^e$ of $T$ in a language $\mathrm{L}^e_{\RV}$. Indeed $T$ has quantifier elimination relative to $\RV$ only if $T^{e}$ does (by \cite[Proposition A.9]{Rid17}). 
    \end{remark}

    \subsubsection{Angular component}\label{SubSecApplAC}
    Let $T$ be a benign theory of Henselian valued fields. We expand the language $\mathrm{L}_{\RV}$ with the sort $k$, $\Gamma$, $\val$, $\res$ and with a new function symbol $\ac: K \rightarrow k$ and consider the theory $T_{\ac}$ of the corresponding valued field with an $\ac$-map. This can be considered as an $\RV$-enrichment of $T$, as the map $\val$ and $\ac$ reduce to $\RV$ (see the diagram below).
    
        \[\xymatrix{
        1 \ar@{->}[r] & O^\times \ar@{->}[r]\ar@{->}[d]_{\res}& K^\star \ar@{->}[r]_\val\ar@{->}[d]_{\rv}\ar@/_1.0pc/@{->}[dl]^{\ac} & \Gamma\ar@{->}[r] \ar@{=}[d]&  0 \\
        1 \ar@{->}[r]&  k^\times\ar@{->}[r]  & \RV^{\star} \ar@/_1.0pc/@{->}[l]^{\ac_{\rv}} \ar@{->}[r]^{\val_{\rv}} & \Gamma \ar@{->}[r]  & 0 }\]  
        
    As Theorem \ref{PropRV} holds in $\RV$-enrichment, we deduce from it the following corollary:
            
    \begin{corollary} \label{CorGammakac}
        Let $\mathcal{L}$ be a model of $T_{\ac}$. \index{Angular component}
        Assume that $\mathcal{L}/\mathcal{K}$ is a separated extension. The following are equivalent:
        \begin{enumerate}
            \item $K$ is stably embedded (resp.\ uniformly stably embedded) in $\mathcal{L}$,
            \item $k_K$ is stably embedded (resp.\ uniformly stably embedded) in $k_L$ and $\Gamma_K$ is stably embedded (resp.\ uniformly stably embedded) in $\Gamma_L$
        \end{enumerate}
    \end{corollary}
    The proof is straightforward, once we prove Lemma \ref{LemmaProductStEmbSubstruc}. Again, it uses Property $(\text{EQ})_{\RV}$ but Property $(\text{Im})$ is not required.

    Given two structures $\mathcal{H}$ and $\mathcal{K}$ (possibly in a different language) with base set $H$ and $K$, we defined the product structure $\mathcal{H}\times \mathcal{K}$ as the three-sorted structure \index{Product of structures} 
    \[\mathcal{H}\times \mathcal{K}= \{H\times K, \mathcal{H},\mathcal{K} \} \cup \{\pi_H: H\times K \rightarrow H,\pi_K: H\times K \rightarrow K\}\]
    where the function symbols $\pi_H$ and $\pi_K$ are interpreted by the canonical projections (for relative quantifier elimination, orthogonality and stable embeddedness, see \cite{Tou20a}).
    One sees that if $H$ and $K$ are two stably embedded and orthogonal definable sets in a structure $\mathcal{M}$, then the product $H \times K$ in $M^2$ with the full induced structure over parameters is isomorphic to $\mathcal{H}\times \mathcal{K}$, where  $\mathcal{H}$ (resp.\ $\mathcal{K}$) is the set $H$ (resp.\ $K$) endowed with its induced structure.

    \begin{lemma}\label{LemmaProductStEmbSubstruc} Let $\mathcal{H}_1$ (resp.\ $\mathcal{K}_1$) be a substructure of a structure $\mathcal{H}_2$ (resp.\ $\mathcal{K}_2$). Then $\mathcal{H}_1\times \mathcal{K}_1$ is a substructure of $\mathcal{H}_2\times \mathcal{K}_2$ and we have:
    \begin{itemize}
        \item $\mathcal{H}_1\times \mathcal{K}_1 \subseteq ^{st} \mathcal{H}_2\times \mathcal{K}_2$ if and only if $\mathcal{H}_1\subseteq ^{st} \mathcal{H}_2$ and $\mathcal{K}_1\subseteq ^{st} \mathcal{K}_2$.
        \item $\mathcal{H}_1\times \mathcal{K}_1 \subseteq ^{ust} \mathcal{H}_2\times \mathcal{K}_2$ if and only if $\mathcal{H}_1\subseteq ^{ust} \mathcal{H}_2$ and $\mathcal{K}_1\subseteq ^{ust} \mathcal{K}_2$.
    \end{itemize}
    \end{lemma}
    \begin{proof}
    The fact that $\mathcal{H}_1\times \mathcal{K}_1$ is a substructure of $\mathcal{H}_2\times \mathcal{K}_2$ is obvious. We prove the non-uniform stable embeddedness transfer. The uniform case can be proved similarly, or can be deduced from the non-uniform case by considering a saturated enough extension in the language of pairs.
    First, we prove the left-to-right implication. This is almost a consequence of purity, but one needs control over the parameters, which is given by relative quantifier elimination. Assume $\mathcal{H}_1\times \mathcal{K}_1 \subseteq^{st} \mathcal{H}_2\times \mathcal{K}_2$. Let $\phi(x_H)$ be a $H_2$-formula, with $x_H$ a tuple of $H$-sorted variables. We identify the formula $\phi(x_H)$ with the formula $\phi(\pi_H(x_{H,K}))$  where $x_{H,K}$ is a variable in $H\times K$(in other words, $\phi(H_2)$ is identified with $\phi(H_2)\times K_2$). By assumption, there is a formula $\psi(x_{H,K},b)$ with parameters $b$ in $H_1 \times K_1$ such that $\psi(H_1\times K_1,b)= \phi(H_1)\times K_1$. By relative quantifier elimination, we may assume that $\psi(x_{H,K},b)$ is a disjunction of formulas of the form:
    \[ \psi_H(\pi_H(x_{H,K}), \pi_H(b)) \wedge \psi_K(\pi_K(x_{H,K}), \pi_K(b))\]
    where $\psi_H(x_H,y_H)$ is a $H$-formula and $\psi_K(x_K,y_K)$ is a $K$-formula.
    It follows that $\phi(H_1)\times K_1$ is a union of rectangles of the forms $\psi_H(H_1, \pi_H(b)) \times  \psi_K(K_1,\pi_K(b))$. The union of the $H$-side gives a formula $\theta(x_H,\pi_H(b))$ with parameters in $H_1$ such that $\theta(H_1,\pi_H(b))=\phi(H_1)$. This shows that $\mathcal{H}_1$ is stably embedded in $\mathcal{H}_2$ and similarly that $\mathcal{K}_1$ is stably embedded in $\mathcal{K}_2$. \\
    Assume $\mathcal{H}_1 \subseteq^{st} \mathcal{H}_2$ and $\mathcal{K}_1 \subseteq^{st} \mathcal{K}_2$. Let $\phi(x_{H,K})$ be a $H_2\times K_2$-formula. Again, by relative quantifier elimination, it is a finite union of rectangles, where the $H$-side has parameters in $H_2$ and the $K$-side has parameters in $K_2$. The trace in $H_1$ of each $H$-side is given by a formula with parameters in $H_1$. Similarly for the $K$-side. At the end, it gives a $H_1\times K_1$-formula $\psi(x_{H,K})$ such that $\psi(H_1\times K_1)=\phi(H_1\times K_1)$.
    \end{proof}
    
    \begin{proof}[Proof of Corollary \ref{CorGammakac}]
    By Theorem \ref{PropRV}, or more precisely by Remark \ref{RemarkReductionToEnrichedRV}, $K$ is stably embedded (resp.\ uniformly stably embedded) in $\mathcal{L}$ if and only if      $\RV_K$ is stably embedded (resp.\ uniformly stably embedded) in the enriched $\RV$-structure 
    \begin{align*}
    \big((\RV_L^\star,\textbf{1},\cdot),(k_L,0,1,+,\cdot),(\Gamma_L,0,+,<),\ac_{\rv}: \RV_L^\star \rightarrow k_L^\star,\\
    \iota: k_L^\star \rightarrow \RV_L^\star, \val_{\rv}:\RV_L^\star \rightarrow \Gamma_L\big).
    \end{align*}
    Notice that, in terms of structure, the injection $\iota: k_L^\star \rightarrow \RV_L^\star$ is superfluous, as is the multiplicative law in $\RV_L^\star$, since the graphs are respectively given by \[\{(a,b)\in k^\star\times \RV^\star \ \vert \ \ac_{\rv}(b)=a \wedge \val_{\rv}(b)=0\}\]
    and 
    \begin{align*}
    \{ ((b_1,b_2),b_3) \in \RV^2\times \RV \ \vert \ &\ac_{\rv}(b_1)\times \ac_{\rv}(b_2)=\ac_{\rv}(b_3) \wedge \\
    &\val_{\rv}(b_1)+ \val_{\rv}(b_2)=\val_{\rv}(b_3) \}.
    \end{align*}
    In other word, $\RV_L$ is exactly the product structure $k_L \times \Gamma_L$: \[\left(\RV_L,(k_L,0,1,+,\cdot),(\Gamma_L,0,+,<),\ac_{\rv}: \RV_L^\star \rightarrow k_L^\star, \val_{\rv}:\RV_L^\star \rightarrow \Gamma_L\right).\]
    Then the corollary is a direct consequence of Lemma \ref{LemmaProductStEmbSubstruc}. \end{proof}

    \subsubsection{Reduction to $\Gamma$ and $k$}\label{SectionGammaK}
    We reduce definability of types over a submodel in $\RV$ to the corresponding conditions in the value group and residue field. We consider the multisorted structure $(\RV_K,\Gamma_K,k_K)$ and its (not necessary elementary) extension $(\RV_L,\Gamma_L,k_L)$. We are going to show that, under some reasonable conditions, $\RV_K$ is stably embedded (resp.\ uniformly stably embedded) in $\RV_{L}$ if and only if $k_K$ is stably embedded (resp.\ uniformly stably embedded) in $k_L$ and  $\Gamma_K$ is stably embedded (resp.\ uniformly stably embedded) in $\Gamma_L$. Recall once again that, if $(\RV,\Gamma,k)$ is an $\RV$-structure, we have the following short exact sequence:
    \[ \xymatrix{1 \ar[r] & k^{\star}\ar[r]& \RV^\star \ar[r] & \Gamma\ar[r] &0}. \]
    It will be seen as a sequence of enriched abelian groups. We will use the quantifier elimination result for short exact sequences of abelian groups due to Aschenbrenner, Chernikov, Gehret and Ziegler (see Paragraph \ref{SubsectionPreliminariesAbelianGroups}). Also, to fit with their notations, let us work in a more general context. Assume we have a (possibly $\{A\}-\{C\}$-enriched) short exact sequence of abelian groups
    \[ \xymatrix{0 \ar[r] & A \ar[r]^{\iota}& B \ar[r]^{\nu} & C\ar[r] &0}, \]
    where $\iota(A)$ is a pure subgroup of $B$. 
    As in Paragraph \ref{SubsectionPreliminariesAbelianGroups}, we see it as an $\mathrm{L}$-structure (resp.\ an $\mathrm{L}_q$ structure), and we denote by $T$ (resp.\ $T_q$) its theory. Let us insist that $A$ and $C$ can be endowed with extra structure.
    By 'the sort $A$' and 'an $A$-formula', we abusively mean respectively the union of sort $\bigcup_{n<\omega}A/nA$ and a  $\bigcup_{n<\omega}A/nA$-formula (so with potentially variables and parameters in $A/nA$).

    We consider the question of stable embeddedness of a sub-short-exact-sequence $\mathcal{M}$ of a model $\mathcal{N}\models T$. We denote by $\bigcup_{n<\omega}\rho_n(\mathcal{M})$ the union of the images of $\mathcal{M}$ in $\bigcup_{n<\omega}A(\mathcal{N}) / nA(\mathcal{N})$ under the maps $\rho_n^{\mathcal{N}}$'s. Notice that it should not be confused with $\bigcup_{n<\omega}A(\mathcal{M}) / nA(\mathcal{M})$.
    
    \begin{proposition} \label{ProAbeGrp} \index{Pure! short exact sequence of abelian groups} \index{Stable embeddedness}
        Let $\mathcal{N}$ be models of $T$  and $\mathcal{M} \subseteq \mathcal{N}$ a sub-short exact sequence of groups. We have:
        \begin{itemize}
        	\item  $\mathcal{M} \subseteq^{st} \mathcal{N}$ if and only if $\bigcup_{n<\omega}\rho_n(\mathcal{M}) \subseteq^{st}  \bigcup_{n<\omega}A/nA(\mathcal{N})$ and $C(\mathcal{M}) \subseteq^{st} C(\mathcal{N})$   and
        	\item  $\mathcal{M} \subseteq^{ust} \mathcal{N}$ if and only if $\bigcup_{n<\omega}\rho_n(\mathcal{M}) \subseteq^{ust}  \bigcup_{n<\omega}A/nA(\mathcal{N})$ and $C(\mathcal{M}) \subseteq^{ust} C(\mathcal{N})$.
\end{itemize}        
    \end{proposition}

    \begin{proof}
    We prove only the non-uniform case. We prove the right-to-left implication first. 
    
    Let $x=(x_0,\ldots,x_{k-1})$ be a tuple of variables in the sort $B$. A term $t(x,b)$ with $b\in B(\mathcal{N})^l$ is of the form $n\cdot x+m\cdot b$ for $n\in \mathbb{N}^k, m\in \mathbb{N}^l$, where $n \cdot x= n_0x_0 + \cdots + n_{k-1}x_{k-1}$.
    By quantifier elimination, it is enough to check that the following formulas define on $M$ some $M$-definable sets:
    \begin{enumerate}
        \item $\phi_C(\nu(t_0(x)),\ldots,\nu(t_{s-1}(x)),c)$ where $t_i(x)$'s are terms with parameters in $B(\mathcal{N})$, $c\in C(\mathcal{N})$ and $\phi_C$ is a $C$-formula,
        \item $\phi_{A}(\rho_{n_0}(t_0(x)),\ldots,\rho_{n_{s-1}}(t_{s-1}(x)),a)$ where $t_i(x)$'s are terms with parameters in $B(\mathcal{N})$, $a\in \bigcup_{n<\omega}A/nA(\mathcal{N})$. 
    \end{enumerate}

    (1) We have $t_l(x)=n_l\cdot x+m_l\cdot b$ where $b$ is some tuple of parameters in $B(\mathcal{N})$. Write $\nu(n\cdot x + m \cdot b))=\nu(n\cdot x)+ \nu(m\cdot b)$. Then $\nu(m\cdot b)$ is a parameter from $C(\mathcal{N})$ and one just needs to apply $C(\mathcal{M}) \subseteq^{st} C(\mathcal{N})$.\\
    (2) Assume $t_0(x)=n\cdot x + m\cdot b$. If for all $g\in B(\mathcal{M})$, $\rho_{n_0}( n \cdot g + m \cdot b)=0$, replace all occurrences of $t_0(x)$ by $0$. Otherwise, for some $g\in B(\mathcal{M})$, we have $\nu(n \cdot g + m \cdot b)\in n_0C$. If  $\nu(n \cdot x + m \cdot b) \in n_0C$, one can write $\rho_{n_0}(n \cdot x + m \cdot b)=\rho_{n_0}(n\cdot x + (-n)\cdot g)+\rho_{n_0}(n \cdot g + m \cdot b)$. The formula is equivalent to 
    \begin{align*}
        \bigg( (\exists y_C \ \nu(n \cdot x+ m \cdot b) = n_0y_C ) \ \wedge \ \phi_{A}(&\rho_{n_0}(n\cdot x + (-n)\cdot g)+\\
        &\rho_{n_0}(n \cdot g + m \cdot b),\ldots,\rho_{n_{s-1}}(t_{s-1}(x)),a) \bigg ) \\
    \vee  \bigg( \neg\exists y_C \ \nu(n \cdot x + m \cdot b) = n_0y_C \wedge \ \phi_{A,n}(&0,\rho_n(t_{1}(x)),\ldots,\rho_n(t_{s-1}(x)),a) \bigg)    
    \end{align*} 
    
    Now $n \cdot x + (-n)\cdot g$ is a term with parameters in $B(\mathcal{M})$ and $\rho_n(n\cdot g + m \cdot b)$ is a parameter in $A/nA(\mathcal{N})$. We proceed similarly for all other terms $ \rho_{n_i}(t_i(x))$, $0<i<s$. As $\bigcup_{n<\omega}\rho_n(\mathcal{M}) \subseteq^{st} \bigcup_{n<\omega}A/nA(\mathcal{N})$, we conclude that the trace in $M$ of the initial formula $\phi_{A}(\rho_{n_0}(t_0(x)),\ldots,\rho_{n_{s-1}}(t_{s-1}(x)),a)$ is the trace in $M$ of a formula with parameters in $M$.
    
    It remains to prove the left-to-right implication. This is almost a consequence of purity of the sorts $A$ and $C$, but one needs control over the set of parameters (we use implicitly \cite[Definition 1.8]{Tou20a}). Let $\phi_A(x_A,a)$ be an $A$-formula with parameters $a\in \bigcup_{n<\omega} A/nA(\mathcal{N})$. By stable embeddedness of $M$ in $\mathcal{N}$, there is an $\mathrm{L}_q$-formula $\psi(x_A,a,b,c)$ with parameters $a,b,c\in A(M)B(M)C(M)$ defining the same set on $A(M)$. As in the proof of Corollary \ref{CoroPureOrth}, we may assume that $\psi(x_A,y_A,y_B,y_C)$ is of the form:
    \[\psi_A(x_A,y_A,\rho_n(t(y_B)))\]
    where $\psi_A$ is an $A$-formula, $t(y_B)$ is a tuple of group terms (and with no occurrence of the variable $y_C$ or of the function symbol $\iota$). 
    This proves that $\bigcup_{n<\omega} \rho_n(M)$ is stably embedded in $\bigcup_{n<\omega} A/nA(N)$. Similarly, we prove that $C(M)$ is stably embedded in $C(\mathcal{N})$.
     \end{proof}
    
    There is a particular case when the condition $\bigcup_{n<\omega}\rho_n(\mathcal{M}) \subseteq^{st}  \bigcup_{n<\omega}A/nA(\mathcal{N})$ (resp.\ $\bigcup_{n<\omega}\rho_n(\mathcal{M}) \subseteq^{ust}  \bigcup_{n<\omega}A/nA(\mathcal{N})$) is equivalent to $A(\mathcal{M}) \subseteq^{st} A(\mathcal{N})$ (resp.\ $A(\mathcal{M}) \subseteq^{ust} A(\mathcal{N}))$, namely when 
    \begin{equation}
        \rho_n(B(\mathcal{M}))= \pi_n(A(\mathcal{M}))  \text{ for all $n\geq 1$.} \label{EquationrhonBpinA}
    \end{equation}
    (notice that the equality $\rho_0(B(\mathcal{M}))= \pi_0(A(\mathcal{M}))= A(\mathcal{M})$ always holds). Indeed,  by bi-interpretability and Remark \ref{RmkTeq}, we have that $A(\mathcal{M}) \subseteq^{st} A(\mathcal{N})$ is equivalent to $\bigcup_{n<\omega}\rho_n(\mathcal{M}) \subseteq^{st}  \bigcup_{n<\omega}A/nA(\mathcal{N})$.
    
    \begin{remark}
    The following conditions implies the condition (\ref{EquationrhonBpinA}):
    \begin{itemize}
        \item $C(\mathcal{M})$ is pure in $C(\mathcal{N})$ (this holds in particular when $\mathcal{M}$ is an elementary submodel of $\mathcal{N}$),
        \item $A(\mathcal{N})/nA(\mathcal{N})$ is trivial for all $n\geq 1$. 
    \end{itemize}
    
    \end{remark}
    The second point is obvious. 
    Let us show that the first one implies (\ref{EquationrhonBpinA}). If $b \in B$ is such that $\rho_n^{\mathcal{N}}(b) \neq 0$, then by definition $\nu(b) \in nC(\mathcal{N})$. By purity, $\nu(b) \in nC(\mathcal{M})$. Then there is $a \in A(\mathcal{M})$ such that $\iota(a)+nB(\mathcal{M})= b + nB(\mathcal{M})$. In particular, $\iota(a)+nB(\mathcal{N})= b + nB(\mathcal{N})$, which means that $\pi_n(a)=\rho_n(b)$. We have showed that $\rho_n(B(\mathcal{M}))= \pi_n(A(\mathcal{M}))$. We conclude by Remark \ref{RmkTeq} that $A(\mathcal{M}) \subseteq^{st} A(\mathcal{N})$ implies $\bigcup_{n<\omega}\rho_n(\mathcal{M}) \subseteq^{st}  \bigcup_{n<\omega}A/nA(\mathcal{N})$.

    We get the following:
    \begin{corollary}\label{CorollaryStablyEmbeddedSubShortExactSequence}
        Let $\mathcal{N}$ be model of $T$  and let $\mathcal{M} \subseteq \mathcal{N}$ be a sub-short-exact-sequence. Assume either 
        \begin{itemize}
            \item that $C(\mathcal{M})$ is a pure subgroup of $C(\mathcal{N})$,
            \item or that $A(\mathcal{N})/nA(\mathcal{N})$ is trivial for all $n\geq 1$.
       \end{itemize}
        Then, we have:
        \begin{itemize}
        	\item  $\mathcal{M} \subseteq^{st} \mathcal{N}$ if and only if $A(\mathcal{M}) \subseteq^{st}  A(\mathcal{N})$ and $C(\mathcal{M}) \subseteq^{st} C(\mathcal{N})$   and
        	\item  $\mathcal{M} \subseteq^{ust} \mathcal{N}$ if and only if $A(\mathcal{M}) \subseteq^{ust}  A(\mathcal{N})$ and $C(\mathcal{M}) \subseteq^{ust} C(\mathcal{N})$.
\end{itemize}
    \end{corollary}
    
        If the reader is not looking for a perfect characterisation, they can easily slightly weaken the second condition in \ref{CorollaryStablyEmbeddedSubShortExactSequence}:

        \begin{corollary}
        Let $\mathcal{N}$ be models of $T$  and let $\mathcal{M} \subseteq \mathcal{N}$ be a sub-short-exact-sequence. Assume
        that $A(\mathcal{N})/nA(\mathcal{N})$ is finite for all $n\geq 1$. 

        Then, we have:
        \begin{itemize}
        	\item $A(\mathcal{M}) \subseteq^{st}  A(\mathcal{N})$ and $C(\mathcal{M}) \subseteq^{st} C(\mathcal{N})$  imply  $\mathcal{M} \subseteq^{st} \mathcal{N}$ and
        	\item  $A(\mathcal{M}) \subseteq^{ust}  A(\mathcal{N})$ and $C(\mathcal{M}) \subseteq^{ust} C(\mathcal{N})$ imply $\mathcal{M} \subseteq^{ust} \mathcal{N}$.
\end{itemize}
    \end{corollary}
    \begin{proof} We deduce from \ref{ProAbeGrp} and from the fact that if $A(\mathcal{N})/nA(\mathcal{N})$ is finite for all $n\geq 1$ that $A(\mathcal{M}) \subseteq^{st}  A(\mathcal{N})$ implies $\bigcup_{n<\omega}\rho_n(\mathcal{M}) \subseteq^{st}  \bigcup_{n<\omega}A/nA(\mathcal{N})$. Indeed, this is due to the fact that the union of a (uniformly) stably embedded set and a finite set is automatically (uniformly) stably embedded.
    \end{proof}
    Combining Theorem \ref{PropRV} and Proposition \ref{ProAbeGrp}, we finally get the following theorem:
    
    \begin{theorem}\label{TheoremReductionGammakNonElementaryPairs}  \index{Valued field! benign Henselian} \index{Stable embeddedness}
    Assume $T$ is a benign theory of Henselian valued fields. 
    Let $\mathcal{L}/ \mathcal{K}$ be a separated extension of valued fields with $\mathcal{L}\models T$. 
    Assume either
    \begin{itemize}
            \item that $\Gamma_K$ is a pure subgroup of $\Gamma_L$,
            \item or that $k_L^\star/(k_L^\star)^n$ is trivial for all $n\geq 1$. 
    \end{itemize}
    
    The following are equivalent:
    
        \begin{enumerate}
            \item $K$ is stably embedded (resp.\ uniformly stably embedded) in $\mathcal{L}$,
            \item $k_K$ is stably embedded (resp.\ uniformly stably embedded) in $k_L$, $\Gamma_K$ is stably embedded (resp.\ uniformly stably embedded) in $\Gamma_L$.
        \end{enumerate}
    \end{theorem}
    Recall that by Proposition \ref{PropRV2}, stably embedded elementary pairs are necessarily separated. As an elementary subgroup is automatically a pure subgroup, we get: 
    
    \begin{theorem}\label{TheoremReductionGammakElementaryPairs}
    Assume $T$ is a benign theory of Henselian valued fields. 
    Let $\mathcal{K} \preceq \mathcal{L}$ be an elementary pair of models of $T$. The following are equivalent:
        \begin{enumerate}
            \item $K$ is stably embedded (resp.\ uniformly stably embedded) in $\mathcal{L}$,
            \item $\mathcal{L}/\mathcal{K}$ is separated, $k_K$ is stably embedded (resp.\ uniformly stably embedded) in $k_L$ and $\Gamma_K$ is stably embedded (resp.\ uniformly stably embedded) in $\Gamma_L$.
        \end{enumerate}
    \end{theorem}

    \subsubsection{Applications}\label{SectionApplications}
    Let us apply Theorem \ref{TheoremReductionGammakNonElementaryPairs} on some examples. We will need:
    
    \begin{fact}[{\cite{Bau80}}]\label{FactMaximalImpliesSeparatedExtension}
     Any extension of a maximal valued field is separated.
    \end{fact}
    Hence, Hahn series $k((\Gamma))$ and Witt vectors $W(k)$ give us a very large branch of examples. We start with the field of $p$-adics:
    
    \begin{corollary}\label{CorollaryQpCp}
    The field of $p$-adics $\mathbb{Q}_p$ is uniformly stably embedded in any algebraically closed valued field containing it. In particular, it is uniformly stably embedded in $\mathbb{C}_p$, the completion of the algebraic closure of $\mathbb{Q}_p$.
    \end{corollary}
    
    \begin{proof}
        By Theorem \ref{TheoremReductionGammakNonElementaryPairs}, it is enough to check that $\mathbb{Z}$ is uniformly stably embedded in any divisible ordered abelian groups containing it (the residue field of $\mathbb{Q}_p$ being finite, there is nothing to prove for the residue field). By an argument similar to that in Remark \ref{RmkPresbArith}, we only have to prove that 1-types over $\mathbb{Z}$ in a divisible ordered abelian group are uniformly definable, which is immediate.  
    \end{proof}
     
    With Theorem \ref{TheoremReductionGammakElementaryPairs}, we recover in particular the theorems of Cubides and Delon (Theorem \ref{ThmCubidesDelon}), and of Cubides and Ye (Theorem \ref{ThmCubidesYe}) in the case of pairs of real closed fields. Here is a list of examples. Notice that some of them are new: 
    
    \begin{examples}
        
        The Hahn series $\mathcal{K}= k((\Gamma))$ where
        \begin{enumerate}
            \item $k\models \text{ACF}_0$, $k= (\mathbb{R},0,1,+,\cdot)$ or $k=(\mathbb{Q}_p,0,1,+,\cdot)$;
            \item $\Gamma = (\mathbb{R},0,+,<)$ or $\Gamma = (\mathbb{Z},0,+,<)$.
        \end{enumerate}
        satisfies $\mathcal{K} \preceq^{ust} \mathcal{L}$ for every elementary extension $\mathcal{L}$ of 
        $\mathcal{K}$.
        \end{examples}

    \begin{proof}
        This is a direct application of Theorem $\ref{TheoremReductionGammakElementaryPairs}$.  Indeed:
        \begin{itemize}
            \item By maximality and Fact \ref{FactMaximalImpliesSeparatedExtension}, all such extensions are separated.
            \item By Fact \ref{FactMS}, $ (\mathbb{R},0,+,<)$ and $(\mathbb{R},0,1,+,\cdot)$ are uniformly stably embedded in any elementary extension.
            \item By Remark \ref{RmkPresbArith}, $(\mathbb{Z},0,+,<)$ is uniformly stably embedded in any elementary extension.
            \item By the work of Cubides and Ye, $\mathbb{Q}_p$ is uniformly stably embedded in any elementary extension.
        \end{itemize}  
    \end{proof}

    A natural example of an algebraically maximal Kaplansky valued field is the valued field
        \[\mathcal{K}=\mathbb{F}_{p}^{alg}(( \mathbb{Z}[1/p] )) \]     
        where $\mathbb{Z}[1/p]$ is the additive group of all rational numbers with denominator a power of $p$. Unfortunately, one can find an elementary extension $\mathcal{L}$ so that $\mathcal{K}$ is not stably embedded in $\mathcal{L}$. In fact, it is the general case: very few valued fields have the property of being stably embedded in any elementary extension, as this property is rare for ordered abelian groups (see Remark \ref{RemarkZROnlyGrupStablyEmbeddedInAnyElementaryExtension}).
        
    \subsection{Unramified mixed characteristic Henselian valued fields}\label{SectionReductionStableEmbeddednessUnramifiedMixedCharacteristic}
        A statement similar to Theorem \ref{PropRV} holds if $\mathcal{L}$ is an unramified mixed characteristic Henselian valued field. As in the case of benign valued fields,  separated pairs are stably embedded if and only if the corresponding pairs of sorts $\RV_{<\omega}$ are stably embedded:

    \begin{theorem} \label{PropRV<omega} Let $\mathcal{L}/ \mathcal{K}$ be a separated extension of valued fields with $\mathcal{L}$ an unramified mixed characteristic Henselian valued field. We see $\mathcal{L}$ as an $\mathrm{L}_{\RV_{<\omega}}$-structure. The following are equivalent:
        \begin{enumerate}
            \item $K$ is stably embedded (resp.\ uniformly stably embedded) in $\mathcal{L}$,
            \item $\RV_{<\omega}(K)$ is stably embedded (resp.\  uniformly stably embedded) in $\RV_{<\omega}(L)$. \index{$\RV_{<\omega}$-sort}\index{Stable embeddedness}
        \end{enumerate}
    \end{theorem}
   The proof of Theorem \ref{PropRV} also holds \textit{mutatis mutandis} for this theorem, where the property $(\text{EQ})_{\RV_{<\omega}}$ is used instead of $(\text{EQ})_{\RV}$. In fact, it holds for any mixed characteristic Henselian valued fields. 
   
   We treat the reduction to the value group and residue field in the more specific case of unramified mixed characteristic Henselian valued fields with perfect residue field. First, we can observe that separatedness is also a necessary condition for elementary pairs to be stably embedded: 
   
   \begin{proposition}\label{PropRV<omega2} \index{Separatedness of valued fields extension}
     Let $\mathcal{L}/ \mathcal{K}$ be an elementary extension of unramified mixed characteristic Henselian valued fields with perfect residue field. If $\mathcal{K}$ is stably embedded in $\mathcal{L}$, then $\mathcal{L}/ \mathcal{K}$ is a separated extension of valued fields.
    \end{proposition}
    The proof of Proposition \ref{PropRV2} also holds in the mixed characteristic case: the only difference is that one should use the properties $(\text{AKE})_{\RV_{<\omega}}$ and  $(\text{Im})_{\RV_{<\omega}}$  instead of respectively the properties $(\text{AKE})_{\RV}$ and  $(\text{Im})$.

   The reduction to the value group and residue field gives a generalisation of the theorem of Cubides and Ye in the case of $p$-adic fields (Theorem \ref{ThmCubidesYe}).
    
      \begin{theorem}\label{TheoremReductionGammakElementaryPairsMixeCharacteristic} 
        Let $\mathcal{K}$ be an unramified mixed characteristic Henselian valued field with perfect residue field, viewed as a structure in the three-sorted language $\mathrm{L}_{\Gamma,k}$ and let $\mathcal{L}$ be an elementary extension. The following are equivalent:
        \begin{enumerate}
            \item $\mathcal{K}$ is stably embedded (resp.\ uniformly stably embedded) in $\mathcal{L}$,
            \item The extension $\mathcal{L}/\mathcal{K}$ is separated, $k_K$ is stably embedded (resp.\ uniformly stably embedded) in $k_L$ and $\Gamma_K$ is stably embedded (resp.\ uniformly stably embedded) in $\Gamma_L$.
            
        \end{enumerate}
    \end{theorem}
    
    Remark that this theorem only treats the case of elementary pairs. The statement like in Theorem \ref{TheoremReductionGammakNonElementaryPairs} in the mixed characteristic case would in particular involve imperfect residue fields. Although quantifier elimination in such a context is known by the work of Anscombe and Jahnke \cite{AJ19}, we chose to not develop this direction. 
    \begin{proof}
        We treat only the non-uniform case. The uniform one is deduced similarly.\\
        $(2 \Rightarrow 1)$. Assume that $k_K \subseteq^{st} k_L$ and $\Gamma_K \subseteq^{st} \Gamma_L$. 
                Recall that we have the following diagrams for all $n \geq 0$:
             \[\xymatrix{
        1 \ar@{->}[r] & O^\times \ar@{->}[r]\ar@{->}[d]_{\res_n}& L^\star \ar@{->}[r]_\val\ar@{->}[d]_{\rv_n} & \Gamma_L\ar@{->}[r] \ar@{=}[d] &  0 \\
        1 \ar@{->}[r]& W_n(k_L)^\times\ar@{->}[r]^{\iota_n}  & \RV_n^{\star}(L) \ar@{->}[r]^{\val_{\rv_n}} & \Gamma_L \ar@{->}[r]  & 0 }\]
        where $W_n(k_L)$ is the truncated ring of Witt vectors of order $n$ over $k_L$. The same diagram holds for $\mathcal{K}$.     Consider an $\RV_{<\omega}$-formula $\phi(x,b)$ with a tuples of variables $x \in \RV_{<\omega}$ and a tuples of parameters $b\in \RV_{<\omega}(L)$. Without loss of generality, we may simplify the notation and assume that $\phi(x,b)$ only quantifies on the sorts $\RV_k$ for $k\leq n$, that all variables $x$ are $\RV_n$-variables and that parameters are all in $\RV_n(L)$. Define the structure $\RV_{\leq n}(L)$ obtained from $\RV_{<\omega}$ by restricting to the sort $\RV_k$, $k\leq n$:
        
\begin{align*}
\RV_{\leq n}(L):=\{&(\RV_{k}(L))_{k\leq n},(W_{k}(k_L),\cdot,+,0,1)_{k\leq n}, (\Gamma_L,+,0,<), \\
& (\val_{\rv_{k}})_{k\leq n},
(\iota_k)_{k\leq n}, (\rv_{m \rightarrow k})_{k<m \leq n} \}
\end{align*}

         with the finite projective system of maps $(\rv_{m\rightarrow k}: \RV_m(L) \rightarrow \RV_k(L))_{k<m\leq n}$. We can also consider the finite projective system of maps $(\res_{m\rightarrow k}: W_m(L) \rightarrow W_k(L))_{k<m\leq n}$. 
         As $\phi(x,b)$ only quantifies on $\RV_{\leq n}$, we have that
         \[\text{For all } a\in \RV_n(L)^{\vert x \vert}, \quad \RV_{<\omega}(L) \models  \phi(a,b) \ \Leftrightarrow \ \RV_{\leq n}(L) \models  \phi(a,b) \]
         
        (notice that $\RV_{\leq n}(L)$ is not endowed with the full induced structure inherited by $\RV_{<\omega}(L)$ as a formula in the induced structure may quantify over the sorts $\RV_N$ with $N>n$). We can see $\RV_{\leq n}(L)$ as a $\{W_n^\times(k_L)\}$-$\{\Gamma_L\}$-enriched exact sequence of abelian groups. Indeed, the kernel of the map $\rv_{m\rightarrow k}$ is given by $1+p^mW_m(k_L)$, a subset of $W_m^\times(k_L)$. It follows that the diagram
        
        \[\xymatrix{
        1 \ar@{->}[r] & W_n^\times(k_L) \ar@{->}[r]^{\iota_n}\ar@{->}[d]_{\res_{n\rightarrow n-1}}& \RV_n(L)^\star \ar@{->}[r]_{\val_{\rv_n}}\ar@{->}[d]_{\rv_{n\rightarrow n-1}} & \Gamma_L\ar@{->}[r] \ar@{=}[d] &  0 \\
        1 \ar@{->}[r]& W_{n-1}^\times(k_L)\ar@{->}[r]^{\iota_{n-1}}\ar@{-->}[d]  & \RV_{n-1}^{\star}(L) \ar@{->}[r]^{\val_{\rv_{n-1}}}\ar@{-->}[d] & \Gamma_L \ar@{->}[r]\ar@{=}[d]  & 0 \\
        1 \ar@{->}[r]& W_1^\times(k_L)\ar@{->}[r]^{\iota_{1}}  & \RV_1^{\star}(L) \ar@{->}[r]^{\val_{\rv_1}} & \Gamma_L \ar@{->}[r]  & 0},\] 
        
        is fully induced by the following diagram:
        
        \[\xymatrix{
        1 \ar@{->}[r] & W_n^\times(k_L) \ar@{->}[r]^{\iota_n}\ar@{->}[d]_{\res_{n\rightarrow n-1}}& \RV_n(L)^\star \ar@{->}[r]_{\val_{\rv_n}} & \Gamma_L\ar@{->}[r] &  0 \\
        & W_{n-1}(k_L)^\times\ar@{-->}[d] & & \\
        & W_1(k_L)^\times& &}.\]
    This means that $\RV_{\leq n}$ is bi-interpretable with the multisorted structure
   \[ \{\RV_{n}(L),(W_{k}(k_L),\cdot,+,0,1)_{k\leq n}, (\Gamma_L,+,0,<), \val_{\rv_{n}} ,
\iota_n, (\res_{m \rightarrow k})_{k<m \leq n}\}.\]
In order to conclude, notice that $W_n(k)$ is interpretable in $k$ for any field $k$ (see e.g. \cite[Corollary 1.64]{Tou20a}). In fact, we see as well that the structure 
\[W_{\leq n}(k):=W_{n}(k) \rightarrow \cdots \rightarrow W_1(k)\]
is interpretable in $k$. 
        Then, we have by Remark \ref{RmkTeq}:
        \[W_{\leq n}(k_K)  \subseteq^{st} W_{\leq n}(k_L). \]

	We can conclude using Corollary \ref{CorollaryStablyEmbeddedSubShortExactSequence} that $\RV_{\leq n}(K)\subseteq^{st} \RV_{\leq n}(L)$. 
	Them, there exists an $\RV_{\leq n}$-formula $\psi(x,c)$ with parameters $c\in \RV_{\leq n}(K)$ such that $\phi(\RV_n(K),b)=\psi(\RV_n(K),c)$. Then, we have
    \begin{align*}
        \text{For all } a\in \RV_n(K)^{\vert x \vert}, \quad \RV_{<\omega}(L) \models  \phi(a,b)        & \Leftrightarrow \ \RV_{\leq n}(L) \models  \phi(a,b) \\
        & \Leftrightarrow \ \RV_{\leq n}(L) \models  \psi(a,c) \\
        &\Leftrightarrow \ \RV_{<\omega}(L) \models  \psi(a,c).
	\end{align*}
	This shows that  $\RV_{<\omega}(K)$ is stably embedded in $\RV_{<\omega}(L)$. We can conclude using Theorem \ref{PropRV<omega}.

Let us prove $(1 \Rightarrow 2)$. Assume that $\mathcal{K}$ is stably embedded in $\mathcal{L}$. By Theorem \ref{PropRV<omega}, we know that $\RV_{<\omega}(K)$ is stably embedded in $\RV_{<\omega}(L)$ and that the extension $\mathcal{L}/\mathcal{K}$ is separated.
We show that $k_K$ is stably embedded in $k_L$. 
Consider a formula $\phi(x,a)$ with free variables $x$ in the residue sort, and a tuple of parameters $a\in k_L$. Again, we can see the residue field $k_L=W_1(k_L)$  as a sort in the structure 
\begin{align*}
\RV_{\leq n}(L):=\{&(\RV_{k}(L))_{k\leq n},(W_{k}(k_L),\cdot,+,0,1)_{k\leq n}, (\Gamma_L,+,0,<), \\
& (\val_{\rv_{k}})_{k\leq n},
(\iota_k)_{k\leq n}, (\rv_{m \rightarrow k})_{k<m \leq n} \}
\end{align*}
and analogously for $k_K$.
As $\RV_{<\omega}(K)$ is stably embedded in $\RV_{<\omega}(L)$, there is for some $n$ an $\RV_{<\omega}$-formula $\psi(x,b)$ with some $\RV_n$-sorted parameters $b\in K$ and quantifiers in $\RV_n$ such that $\psi(k_K^{\vert x \vert},b)=\phi(k_K^{\vert x \vert},a)$. 
As before, we observe that $\RV_{\leq n}:= \cup_{m\leq n} \RV_m$ is a $\{W_n^\times(k)\}$-$\{\Gamma\}$-enrichment of the exact sequence of abelian groups:
\[1 \rightarrow W_n^\times(k) \rightarrow \RV_{n} \rightarrow \Gamma \rightarrow 0.\]
We may apply Fact \ref{FactACGZ}: there is a formula $\theta(x,y)$ in the language  
\begin{align*}
    \{(W_{k}(k_L),\cdot,+,0,1)_{k\leq n}, (W_n^\times/(W_n^\times)^m)_{m>1}, (\pi_m:W_n^\star \rightarrow W_n^\star/(W_n^\star)^m)_{m>1}, \\
    (\res_{l \rightarrow k})_{k<l \leq n} \}
\end{align*} 

and group-sorted terms $t(y)$ such that 
\[ \theta(k_K^{\vert x \vert},\rho_m(t(b)))= \psi(k_K^{\vert x \vert},b) = \phi(k_K^{\vert x \vert},a)\]
for some $m<\omega$, where $\rho_m$ is defined as in Paragraph \ref{SubsectionPreliminariesAbelianGroups}. As the structure
\[W_{\leq n}(k):= W_n(k)\rightarrow W_{n-1}(k) \rightarrow \cdots \rightarrow W_1(k)\]
is interpretable in $k$ for all fields $k$, we have a formula $\xi(x,c)$ in the language of rings and with parameters $c$ in $k_K$ such that \[\xi(k_K^{\vert x \vert},c)=\theta(k_K^{\vert x \vert},\rho_m(t(b)))= \psi(k_K^{\vert x \vert},b) = \phi(k_K^{\vert x \vert},a).\]
This shows that $k_K$ is stably embedded in $k_L$. Similarly, we show that $\Gamma_k$ is stably embedded in $\Gamma_L$.
    \end{proof}
    
    \begin{example}
    Consider  $\mathcal{K}=Q(W(\mathbb{F}_p^{alg}))$ \index{Ring of Witt vectors} the quotient field of the ring of Witt vectors over the algebraic closure of $\mathbb{F}_p$. Then $\mathcal{K}$ is uniformly stably embedded in any elementary extension. 
    \end{example}
   \begin{proof}
    By Fact \ref{FactMaximalImpliesSeparatedExtension}, any elementary extension $\mathcal{L}/\mathcal{K}$ is separated. The value group $\mathbb{Z}$ of $\mathcal{K}$ is stably embedded in any elementary extension by Remark \ref{RemarkZROnlyGrupStablyEmbeddedInAnyElementaryExtension}, and its residue field is even stable. We can conclude with Theorem \ref{TheoremReductionGammakElementaryPairsMixeCharacteristic}.  
   \end{proof}
    
    \subsection{Axiomatisability of stably embedded pairs}\label{SectionPairs}
        
    Closely related to the question of definability of types is the question of elementarity of the class of stably embedded pairs. Consider $T$ a theory in a language $\mathrm{L}$ and denote by $T_P$ in the language $\mathrm{L}_P=\mathrm{L}\cup \{P\}$ \nomenclature[L]{$\mathrm{L}_P=\mathrm{L}\cup \{P\}$}{} the theory of elementary pairs $\mathcal{M} \preceq \mathcal{N}$ where $P$ is a predicate for $\mathcal{M}$. We consider the following subclasses of models:
    \[\mathcal{C}^{st}_T=\{(\mathcal{N},\mathcal{M}) \ \vert \ \mathcal{N} \succeq \mathcal{M}\models T, \ \mathcal{M}\text{ is stably embedded in }\mathcal{N}\}\]
and 
    \[\mathcal{C}^{ust}_T=\{(\mathcal{N},\mathcal{M}) \ \vert \ \mathcal{N} \succeq \mathcal{M}\models T, \ \mathcal{M}\text{ is uniformly stably embedded in }\mathcal{N}\}.\]
    
    As Cubides and Ye in \cite{CY19}, we are asking whether these classes are first order. If that is the case, we denote respectively by $T^{st}_P$ and $T^{ust}_P$ their respective theory, and we say respectively that $T^{st}_P$ and $T^{ust}_P$ exist.

    \begin{remark}\label{RemarkTPstImpliesTPUST} If $\mathcal{M}\preceq^{st} \mathcal{N}$ is a stably embedded elementary pair of models which is not uniformly stably embedded, then any $\vert \mathrm{L} \vert$-saturated elementary extension of the pair is not stably embedded.
    \end{remark}
    \begin{proof}
    	By assumption, there is a formula $\phi(x,y)$ such that for all formulas $\psi(x,z)$, there is a $b\in N^{\vert y \vert}$ such that for any $c\in M^{\vert z \vert}$, $\phi(M^{\vert x \vert},b) \neq \psi(M^{\vert x \vert},c)$. By a usual coding trick, we have in fact that for any finite set of formulas $\Delta$, there is a $b\in N^{\vert y \vert}$ such that for any $\psi(x,z)\in \Delta$ and $c\in M^{\vert z \vert}$, $\phi(M^{\vert x \vert},b) \neq \psi(M^{\vert x \vert},c)$. This is to say that the type in the language of pairs
    	\[p(y)=\{ \forall c\in M^{\vert z \vert}, \phi(M^{\vert x \vert},y) \neq \psi(M^{\vert x \vert},c)\}_{\psi(x,z) \in \mathrm{L}}\]
    	is finitely satisfiable in $(\mathcal{N}, \mathcal{M})$. Thus, it is realised in any $\vert \mathrm{L} \vert$-saturated elementary  extension of the pair. Such a pair will not be stably embedded.
    \end{proof}
    
    \begin{corollary}
        If $T^{st}_P$ exists, then $T^{ust}_P$ exists and $T^{ust}_P=T^{st}_P$. 
    \end{corollary}
    
    In a stable theory $T$, elementary submodels are always stably embedded, $\mathcal{C}^{st}_T$ and $\mathcal{C}^{ust}_T$  are simply the class of elementary pairs. In other words, we have $T^{st}_P=T^{ust}_P=T_P$. Let us analyse few more examples. We quickly cover the case of o-minimal theories and the Presburger arithmetic. The reader will find in Appendix \ref{OnPairsOfRandomGraphs} the case of the theory of random graphs. 
     \newline   
    
    \textbf{O-minimal theories}
    
        Consider $T$ an o-minimal theory.  By Marker-Steinhorn, an elementary pair $(R_1,R_0)$ of models of $T$ is stably embedded if and only if it is uniformly stably embedded, if and only if all elements in $R_1$ either realise a type at the infinity or a rational cut. So both $T^{st}_{P}$ and $T^{ust}_{P}$ exist, and $T^{st}_{P} = T^{ust}_{P}$. This theory is given by the theory $T_P$ together with the axiom
        \begin{align*}
         \forall x \notin P \ ( \forall b\in P \ x>b) \vee (\forall b\in P \ x<b)\\
          \vee( \exists a\in P \ a<x \ \wedge \ \forall b\in P \ (a<b)\Rightarrow (x<b)) \\
          \vee( \exists a\in P \ x<a \ \wedge \ \forall b\in P\ (b<a)\Rightarrow (b<x)).
          \end{align*}
    
    \textbf{Presburger arithmetic}\\
        Consider the theory $T$ of $(\mathbb{Z}, 0 ,1, +, <)$. An elementary submodel $\mathcal{Z}$ of a model $\mathcal{Z}'$ is stably embedded if and only if no element in $\mathcal{Z}'$ realises a proper cut. So both $T^{st}_{P}$ and $T^{ust}_{P}$ exist and are equal. This theory is given by the theory $T_P$ together with the axiom saying that $P$ is convex.

    \subsubsection{Benign theories of Henselian valued fields}
    Let $T$ be a completion of a benign theory of Henselian valued fields of equicharacteristic. We denote by $T_{\Gamma}$ and $T_k$ the corresponding theories of the value group and residue field. We require here $T$ to be of equicharacteristic in order to get a canonical maximal Henselian valued field of given value group and given residue field, namely the Hahn series: if $\Gamma \models T_\Gamma$ and $k\models T_k$ then $k((\Gamma))\models T$.
    

    We have the following reduction:
    \begin{proposition}
    	\begin{itemize}
    		\item $T_P^{st}$  exists if and only if $(T_{\Gamma})_P^{st}$  and $(T_{k})_P^{st}$  exist,
         	\item $T_P^{ust}$  exists if and only if $(T_{\Gamma})_P^{ust}$  and $(T_{k})_P^{ust}$  exist.
    \end{itemize}
        
    \end{proposition}
    Notice that the case of algebraically closed  valued fields (including ones of mixed characteristic) follows from the work of Cubides and Delon. The case of p-adically closed valued fields (which is not covered by this theorem) is treated together with the case of real closed valued fields in the work of Cubides and Ye. See \cite[Theorem 4.2.4.]{CY19}. Here is a list of examples (some of them are new):

    \begin{examples}

      Let $T$ be the theory of the Hahn series $\mathcal{K}= k((\Gamma))$ where
        \begin{enumerate}
            \item $k\models \text{ACF}_0$, $k= (\mathbb{R},0,1,+,\cdot)$ or $k=(\mathbb{Q}_p,0,1,+,\cdot)$;
            \item $\Gamma \models \text{DOAG}$ or $\Gamma\models \Th(\mathbb{Z},0,+,<)$.
        \end{enumerate}
        Then $T_P^{st}$ and $T_P^{ust}$ exist. 
    
    \end{examples}

    \begin{proof}
        We prove the non-uniform case. The right-to-left implication is an easy consequence of Theorem \ref{TheoremReductionGammakElementaryPairs}. Indeed one has that the class $\mathcal{C}^{st}_T$ is axiomatised by:
        \[((T_{\Gamma}
        )_P^{st} \cup (T_k)_P^{st} \cup \{K/P(K) \text{ is separated}\}.\]
        Assume $T_P^{st}$ exists. We prove that the class $\mathcal{C}^{st}_{T_k} =\{(k_1,k_2) \ \vert \ k_1 \preceq^{st} k_2 \models T_k,\}$ is axiomatisable. One can show that $\mathcal{C}^{st}_{T_\Gamma}$ is axiomatisable by the same argument. By Theorem \ref{TheoremReductionGammakElementaryPairs}, we have 
        \[\mathcal{C}^{st}_{T_k} =\{(k_1, k_2) \ \vert \ \text{there is }\Gamma \models T_{\Gamma} \ \text{such that } k_1((\Gamma)) \preceq^{st} k_2((\Gamma)) \models T, \ k_1,k_2 \models T_k\}.\]
        Indeed, $T$ admits as models these Hahn series fields, and Hahn series are always maximal, so in particular every extension of such is separated (Fact \ref{FactMaximalImpliesSeparatedExtension}).  This class is closed under ultraproducts: let $I$ be a set of indices and $(k_1^i,k_2^i)\in \mathcal{C}^{st}_k$ for all $i\in I$, with the corresponding $\Gamma^i \models T_\Gamma$. Let $\mathcal{U}$ be an ultrafilter on $I$. Then by Theorem \ref{TheoremReductionGammakElementaryPairs}, we have 
        \[\prod_{\mathcal{U}}(k_1^i((\Gamma^i))) \preceq^{st} \prod_{\mathcal{U}}(k_2^i((\Gamma^i))).\]
        as $k_1^i((\Gamma^i)) \preceq^{st} k_2^i((\Gamma^i))$ for all i and as $\mathcal{C}^{st}_{\Th(K)}$ is closed by ultraproduct.
        Again by Theorem \ref{TheoremReductionGammakElementaryPairs}, we have: 
        \[\prod_{\mathcal{U}}(k_1^i) \preceq^{st} \prod_{\mathcal{U}}(k_2^i),\]
        (the ultraproduct commutes with the residue map).
        Obviously, $\mathcal{C}^{st}_{\Th(k)}$ is stable under isomophism and if $(k_2,k_1)\preceq (k_2',k_1')$ with $k_1' \preceq^{st} k_2'$, then $k_1 \preceq^{st} k_2$ (Remark \ref{Rmkstnonuniform}). This proves that $\mathcal{C}^{st}_{\Th(k)}$ is closed under elementary equivalence and, finally, that it is axiomatisable.
    \end{proof}
    

    \subsubsection{Bounded formulas and elimination of unbounded quantifiers}
    Let $T$ be a complete first order theory in a language $\mathrm{L}$.
    \begin{definition}
        We say that an $\mathrm{L}_P$-formula is bounded if it is of the form: 
        \[\quantifies_0y_0 \in P\ \cdots \quantifies_{n-1}y_{n-1} \in P \  \phi(x,y_0,\ldots,y_{n-1}),\]
        where $\phi(x,y_0,\ldots,y_{n-1})$ is an $\mathrm{L}$-formula and $\quantifies_0,\ldots, \quantifies_{n-1}\in \{\forall, \exists \}$.
        We say that a theory extending the theory of pair $T_P$ eliminates unbounded quantifiers when any formula is equivalent to a bounded formula.
    \end{definition}
    
    Let us cite here examples of elimination of unbounded quantifiers.  
    \begin{examples}
    \begin{itemize}
        \item Assume that $T$ is an o-minimal theory extending the theory of ordered abelian groups with distinguished positive element 1. The theory of dense pairs of models of $T$ eliminates unbounded quantifiers (\cite[Theorem 2.5]{VdD98}).
        \item Assume that $T$ is stable. The theory of belle pairs $T'\supseteq T_P$ eliminates unbounded quantifiers if and only if $T$ does not have the finite cover property (see \cite[Theorem 6]{Poi83}). 
        \item Assume that $T$ is NIP. Let $I \subseteq \mathcal{M}$ be an indiscernible sequence indexed by a dense complete linear order so that every type over $I$ is realised in $\mathcal{M}$. Then $\Th(\mathcal{M},\mathcal{I})$ is bounded ({ \cite[Theorem 3.3]{BB00}}).

    \end{itemize}

    
    \end{examples}
        The reader will find an overview on elimination of unbounded quantifier in \cite{CS13}. We give now another characterisation of stable embeddedness modulo elimination of unbounded quantifiers. 
    
    \begin{proposition}
        Let $\mathcal{M} \preceq \mathcal{N}$ be two models of $T$. Assume that the theory $\Th(\mathcal{N},\mathcal{M})$ in the language of pairs $\mathrm{L}_P$ eliminates unbounded quantifiers. Then $\mathcal{M}$ is uniformly stably embedded in $\mathcal{N}$ if and only if $P$ is a (uniformly) stably embedded predicate for $\Th(\mathcal{N},\mathcal{M})$.
    \end{proposition}
    
    \begin{proof} Assume $\mathcal{M}$ to be uniformly stably embedded in $\mathcal{N}$, and let $(\mathbf{N},\mathbf{M})$ be a monster model. So $\mathbf{M}$ is stably embedded in $\mathbf{N}$. Let $\phi(x,a)$ be an $\mathrm{L}_P$-formula with parameter $a\in \mathbf{N}$. It is equivalent to a bounded formula
   \[\quantifies_0y_0 \in P\ \cdots\ \quantifies_{n-1}y_{n-1} \in P\ \psi(x,y_0,\ldots,y_{n-1},a),\]
        with $\psi(x,y_0,\ldots,y_{n-1},z)$ an $\mathrm{L}$-formula and $\quantifies_0,\ldots, \quantifies_{n-1}\in \{\forall, \exists \}$.
    Then the definable set $\psi(\mathbf{M}^{n+1},a)$ is given by some  $\theta(\mathbf{M}^{n+1},b)$ where $\theta(x,y_0,\ldots,y_{n-1},z)$ is an $\mathrm{L}$-formula and $b \in \mathbf{M}^{\vert z \vert}$. Hence $ P(x) \wedge \phi(x,a) $ is equivalent to
  \[P(x) \wedge \quantifies_0y_0\ \cdots \quantifies_{n-1}y_{n-1} \theta(x,y_0,\ldots,y_{n-1},b) .\]
    This proves that $P(x)$ is a stably embedded predicate in $\Th(\mathcal{N},\mathcal{M})$. \\

    Assume that $P$ is a (uniformly) stably embedded predicate. We want to show that $\mathcal{M}$ is uniformly stably embedded in $\mathcal{N}$. It is enough to show that if $(\mathcal{N}',\mathcal{M}') \succeq (\mathcal{N},\mathcal{M})$, then $\mathcal{M}'$ is stably embedded in $\mathcal{N}'$ (see Remark \ref{RmqUstMonsterModel}). Let $(\mathcal{N}',\mathcal{M}')$ be an elementary extension of the pair and let $\phi(x,a)$ be an $\mathrm{L}(\mathcal{N}')$-formula. As the predicate $P$ is stably embedded, the set $\phi(\mathcal{M}',a)$ is given by $\psi(\mathcal{M}',b)$ where $\psi(x,z)$ is an $\mathrm{L}_P$-formula with parameters $b\in \mathcal{M}'$. By elimination of unbounded quantifier, we may assume that $\psi(x,z)$ is of the form 
       \[\quantifies_0 y_0 \in P\ \cdots\ \quantifies_{n-1} y_{n-1} \in P\ \theta(x,y_0,\ldots,y_{n-1},z).\]
    Replace all bounded quantifier over $P$ by the corresponding unbounded quantifier. We obtain an $\mathrm{L}$-formula $\psi'(x,z)$ such that $\psi'(\mathcal{M}',b)=\psi(\mathcal{M}',b)$. This proves that $\mathcal{M}'$ is stably embedded in $\mathcal{N}'$.
    \end{proof}
\appendix

\section{On pairs of random graphs}\label{OnPairsOfRandomGraphs}
We construct in this appendix two stably embedded pairs of random graphs: one is non-uniform and the other is uniform. This gives another illustration of the difference between these notions. We then show that the class of uniformly (resp.\  non-uniformly) stably embedded pairs of random graphs is not elementary.   We also take the occasion to talk about quantifier elimination in the Shelah expansion of a random graph.

    Let $T$ be the theory of the random graph in the language $\mathcal{L}=\{R\}$ and $\mathcal{G}$ a model of $T$. The Marker-Steinhorn criterion holds:
    \begin{remark}
        Let $\bar{a}=a_0,\ldots a_{n-1}$ be a finite tuple of elements in an elementary extension $\mathcal{H}$ of $\mathcal{G}$. Then, by quantifier elimination, one has
        \[\tp(\bar{a}) \cup \bigcup_{i<n} \tp(a_{i}/G)  \vdash \tp(\bar{a}/G).\]
        It follows that for all $n \in \mathbb{N}$, \[T_1(\mathcal{G},\mathcal{H}) \Rightarrow T_n(\mathcal{G},\mathcal{H}),\]
        and
        \[T_1^u(\mathcal{G},\mathcal{H}) \Rightarrow T_n^u(\mathcal{G},\mathcal{H}).\]
    \end{remark}

    We will give two constructions of an elementary extension $\mathcal{H}$ of $\mathcal{G}$ such that all types over $G$ realised in $H$ are definable. The second one is simpler but not uniform. We will give later a generalisation of the first construction. \\
    
\textbf{Construction (H1).}\label{ConstructionH1} Assume $\lambda= \vert G\vert$. We are going to build an increasing sequence $(\mathcal{G}_{j})_{j<\omega}$ of graphs containing $\mathcal{G}$ and of cardinality $\lambda$. 
    We set $\mathcal{G}_0:=\mathcal{G}$. Assume that for some $j<\omega$, $\mathcal{G}_j$ has been constructed, and let $(A_i,B_i)_{i<\lambda}$ be an enumeration of pairs of finite subsets of $G_j$ such that for all $i<\lambda$, $A_i \cap B_i =\emptyset$. For each $i<\lambda$, pick any $m_i \in G$  such that $m_i$ is related to $A_i \cap G$ and unrelated to $B_i\cap G$. Then, for each $i<\lambda$, we consider a new point $\delta_{j}^i$. We set $G_{j+1}=G_j \cup \{\delta_{j}^i\}_{i<\lambda}$ and $R(\delta_{j}^i,g)$ if and only if ($g\in G$ and $R(m_i,g)$) or $g\in A_i$.
    Finally, we set $\mathcal{H}_1=\bigcup_{j<\omega} \mathcal{G}_j$. 
    
    \begin{remark}
        $\mathcal{H}_1$ is by construction a random graph containing $\mathcal{G}$. Also, every $1$-type over $G$ realised in $H_1 \setminus G$ is of the form:
         \[m_{\neq} :=< \{R(x,g) \ \vert \ g\in R(m,G)\} \cup \{x \neq g \ \vert \ g\in G\} >\]
         where $m\in G$. Set of such 1-types is uniformly definable. Hence, one has $\mathcal{G} \preceq^{ust} \mathcal{H}_1$.
    \end{remark}    \textbf{Construction (H2).} Assume $\lambda= \vert G\vert$. We are going to create an increasing sequence $(\mathcal{G}_{j})_{j<\omega}$ of graphs containing $\mathcal{G}$ and of cardinality $\lambda$. 
    Set $\mathcal{G}_0:=\mathcal{G}$ and assume that $\mathcal{G}_j$ has been constructed. Let $(A_i)_{i\in \lambda}$ be an enumeration of finite subsets of $G_j$. For all $i<\lambda$, we consider a new point $(\delta_j^i)$. We set $G_{j+1}=G_j \cup \{\delta_{j}^i\}_{i<\lambda}$ and $R(\delta_{j}^i,g)$ if and only if $g\in A_i$.
    Finally, we set  $\mathcal{H}_2=\bigcup_{j<\omega} \mathcal{G}_j$.
    
    \begin{remark}\label{RmkRG}
    $\mathcal{H}_2$ is by construction a random graph containing $\mathcal{G}$. Also, every $1$-type over $G$ realised in $H_2 \setminus G$ is of the form:
         \[p_A :=<\{R(x,g) \ \vert \ g\in A\} \cup \{\neg R(x,g) \ \vert \ g\notin A\} \cup \{x \neq g \ \vert \ g\in G\}>\]
         where $A$ is a finite subset of $G$. Such 1-types are definable. Hence one has $\mathcal{G} \preceq^{st} \mathcal{H}_2$.
    \end{remark}
    One can easily see that $G$ is not uniformly stably embedded in $H_2$. Indeed, by construction, for all $h\in H_2 \setminus G$, $R(h,G)= \{g\in G \ \vert  \ R(h,g)\}$ is finite and conversely, any finite set is given by $R(h,G)$ for some $h$. If it were uniformly defined by another formula with parameters in $G$, this would contradict the fact that $T$ eliminates $\exists^\infty$ . 
    \\

        \textbf{No axiomatisation of stably embedded pairs of random graphs}\\
        We show that the classes of (uniformly) stably embedded pairs of random graphs is not axiomatisable in the language of pairs $\mathrm{L}_P=\mathrm{L}\cup \{P\}$. Let us recall the notation of Section \ref{SectionPairs}:
            \[\mathcal{C}^{st}_{T}=\{(\mathcal{N},\mathcal{M}) \ \vert \ \mathcal{N} \succeq \mathcal{M}\models T, \ \mathcal{M}\text{ is stably embedded in }\mathcal{N}\}\]
and 
    \[\mathcal{C}^{ust}_{T}=\{(\mathcal{N},\mathcal{M}) \ \vert \ \mathcal{N} \succeq \mathcal{M}\models T, \ \mathcal{M}\text{ is uniformly stably embedded in }\mathcal{N}\}.\]

    We show that $\mathcal{C}^{st}_{RG}$ and $\mathcal{C}^{ust}_{RG}$ are not axiomatisable. By Remarks \ref{RmkRG} and \ref{RemarkTPstImpliesTPUST}, $\mathcal{C}^{st}_{RG}$ is not preserved by elementary equivalence: using the notation of Paragraph \ref{PreliminariesOnDefinableTypes}, we have that $\mathcal{G} \preceq^{st} \mathcal{H}_2$, but in any saturated enough extension $(\mathcal{G}',\mathcal{H}_2')$ of the pair $(\mathcal{G},\mathcal{H}_2)$, $\mathcal{G}'$ is not stably embedded in $\mathcal{H}_2'$.
    
    However, $\mathcal{C}^{ust}_{RG}$ is preserved by elementary extensions (as "for all $y$, there is a $z$ such that $\phi(M^{\vert x\vert},y)=\psi(M^{\vert x\vert},z)$" is a first order property in the language of pairs of the formulas $\psi(x,z)$ and $\phi(x,y)$). Let us show that $\mathcal{C}^{ust}_{RG}$ is not stable by ultraproduct. We will need few facts about random graphs. We leave proofs (by induction) to the reader.
    
    Let us denote by $\phi_{n,m}(x,y_0,\ldots, y_{m-1},z_0, \ldots z_{n-1})$ the formula 
    \[ \bigwedge_{i<n}  R(x,y_i) \wedge \bigwedge_{j<m}\neg R(x,z_j),\]
    for $n,m$ positive integers, $n+m > 0$ and with all variables
    distinct. Let $\Phi$ be the set of such formulas.
    \begin{fact}\label{FactRGSeparate}
        Let $\mathcal{G}$ be a random graph and let $n$ and $m$ be positive integers, $n+m > 0$. For any disjoint finite subsets $A$ and $B$ of $G$, there is an instance of $\phi_{n,m}(x,g)$, with parameters $g=(g_0,\ldots,g_{m+n-1})\in G^{n+m}$ all distinct such that $A \subseteq \phi_{n,m}(G,g) \subseteq B^{\mathsf{C}}$.
    \end{fact}
    
    The following fact says that an instance of $\phi_{n,m}$ is `bigger' than instances of $\phi_{n',m'}$ if $n+m<n'+m'$.
    \begin{fact}\label{FactInfiniteCofinfinite}
        Let $\mathcal{G}$ be a random graph.
        \begin{itemize}
            \item Let $\phi(x,y)\in \Phi$, and let $a\in G^{\vert y \vert}$. Then $\phi(G,a)$ is infinite and co-infinite. 
            \item Consider some distinct parameters $g=(g_0,\ldots, g_{n-1},g_0',\ldots, g_{m-1}') \in G^{n+m}$ and some distinct parameters $h=(h_0,\ldots, h_{n'-1},h_0',\ldots, h_{m'-1}') \in G^{n'+m'}$. Assume $n'+m'>n+m$. Then
            \[ \phi_{n,m}(G,g) \setminus \phi_{n',m'}(G,h) \]
            is infinite.
        \end{itemize}
        
    \end{fact}
    As we did previously with Construction $\mathcal{H}_1$, we construct an extension $\mathcal{H}_{n,m}$ such that the set of traces $\{R(G,h) \ \vert \ h\in H_{n,m}  \}$ is contained in 
    \[\{\phi_{n,m}(G,g)\ \vert \ g=(g_0,\ldots, g_{n-1},g_0',\ldots, g_{m-1}') \in G^{n+m} \text{ with all $g_i,g_j'$ distinct}\},\]
     the set of instances of $\phi(x,y)$ with distinct parameters in $G$ (we use Fact \ref{FactRGSeparate}). If $\mathcal{U}$ is a non-principal ultra-filter in $\mathbb{N}\times \mathbb{N}\setminus \{(0,0)\}$, let us denote by $(\Tilde{\mathcal{H}},\Tilde{\mathcal{G}})$ the ultraproduct $\prod_{\mathcal{U}}(\mathcal{H}_{n,m},\mathcal{G})$. One sees that any set $R(\Tilde{G},a)$ for $a \notin \Tilde{G}$ is infinite and co-infinite in $\Tilde{G}$. But it cannot be given, even up to a finite set, by a (positive) boolean combination of  instances of formulas in $\Phi$ with parameters in $\Tilde{G}$. Indeed, consider such an instance $\phi_{n,m}(x,g)$. Then, by Fact \ref{FactInfiniteCofinfinite}, the projection of $\phi_{n,m}(\tilde{G},g) \setminus R(\tilde{G},a)$ to $\mathcal{H}_{n',m'}$ is infinite if $n'+m'> n+m$. We conclude by Łoś's theorem (and the fact that $\Phi$ is closed under intersection). It follows that $\Tilde{\mathcal{G}}$ is not stably embedded in $\Tilde{\mathcal{H}}$.

    \textbf{More on expansions of random graphs:} 
    We briefly show that in any model $\mathcal{G}$ of the random graph, the Shelah expansion $\mathcal{G}^{\text{sh}}$ (see e.g. \cite[Definition 3.9]{Sim15}) does not eliminate quantifiers. 
    \begin{claim}
            There is a clique $A$ of $G$ consisting of (distinct) elements $a_{0,i},a_{1,i} \ i<\omega$ such that the subset $E=\{(a_{0,i},a_{1,i})\}_{i<\omega}$ of $G^2$ is definable in $\mathcal{G}^{sh}$.
    \end{claim}
        
        \begin{proof}
            First notice that any subset of $G$ is externally definable. We construct by induction the subset $A=\{a_{0,i},a_{1,i}\}_{i<\omega}$ and a subset $B=\{b_{0,i},b_{1,i}\}_{i<\omega}$ of $G$ such that: 
            \begin{itemize}
                \item $A$ is a clique: for all $i,j<\omega$, $\xi,\zeta\in \{0,1\}$, $a_{\xi,i}R \ a_{\zeta,j}$
                \item for all $i,j<\omega$ and $\xi,\zeta\in \{0,1\}$,  $a_{\xi,i}R \ b_{\zeta,j}$ if and only if  $i=j$ and $\xi \leq \zeta$.
            \end{itemize}
            The set $B$ `encodes' that the pairs $(a_{0,i},a_{1,i})$ belong to $E$. Indeed, a definition of $E$ will be given by:
            \[(a_0,a_1)\in {A}^2 \wedge \exists b_{0},b_1 \in B \ a_{i}R \ b_{j} \ \Leftrightarrow \ i\leq j.\]
        
            For some $k<\omega$, assume that $(a_{0,i},a_{1,i}, b_{0,i},b_{1,i})_{i<k}$ has been constructed.
            Take $a_{0,k},a_{1,k}$ any elements of $G$ such that:
            \begin{itemize}
                \item $a_{0,k}R a_{1,k}$
                \item for all $i<k$ and $\xi,\zeta\in \{0,1\}$, $a_{\xi,k}R\ a_{\zeta,i}$.
                \item for all $i<k$, and $\xi,\zeta\in \{0,1\}$, $\neg a_{\xi,k}R b_{\zeta,i}$.
            \end{itemize}
            Then take $b_{0,k}$ and $b_{1,k}$ in $G$ such that for all $i<k$ and $\xi,\zeta\in \{0,1\}$,  $a_{\xi,i}R\  b_{\zeta,k}$ if and only if  $i=k$ and $\xi \leq \zeta$.  This conclude our induction.
        \end{proof}
                It is easy to see that the set $E$ is not externally definable. Indeed, as $A$ is a clique, the only externally definable subsets of $A^2$ are --by quantifier elimination -- Boolean combinations of rectangles $S\times S'$ (where $S,S'\subseteq A$) and of the diagonal $\{(a,a)\ \vert \ a \in A\}$. Note also that the class of externally definable subsets is closed under Boolean combinations. The proposition follows:
        \begin{proposition}
            For any random graph $\mathcal{G}$, the Shelah expansion $\mathcal{G}^{\text{sh}}$ does not eliminate quantifiers.
        \end{proposition}
        One can also ask if adding predicates for all subsets of $G^n$ for a given $n$ gives us quantifier elimination. For the same reason, this does not hold either: one can similarly show that their is a definable subset $ E_n:=\{(a_{0,i},\ldots,a_{n,i})\ \vert \ i<\omega  \}$ of ${(G^{sh})}^{n+1}$ where all $a$'s are distinct and where $A=\{a_{0,i},\ldots,a_{n,i}, \ i<\omega\}$ is a clique. Then, one sees that this set cannot be written as a finite Boolean combination of rectangles $S^{n_1}\times \cdots \times S^{n_k}$ with $2\leq k\leq n+1$, $\sum_{i<k}n_i=n+1$ and $S^{n_i}\subseteq A^{n_i}$, and the diagonal $\{(a,\ldots,a) \in A^{n+1}\}$. 

\newpage
\begin{acknowledgements}

    I want to thank my PhD advisor Martin Hils for his support and his guidance. Many thanks to Pablo Cubides for his enlightenment of his work in \cite{CD16}. In particular, my gratitude to both of them for sharing with me their ideas which led to the proof of Proposition \ref{PropRV2} and motivated the writing of this paper. Many thanks to Artem Chernikov and Martin Ziegler for sharing their results in \cite{ACGZ20}. Finally, many thanks to Martin Bays and Allen Gehret, whose help and discussions have greatly improved the quality of this paper.
\end{acknowledgements}

\bibliographystyle{plain}
\bibliography{Bibtex}

\end{document}